\documentclass[reqno, twoside, a4paper,10pt]{amsart}

\usepackage[czech]{babel}
\usepackage[cp1250]{inputenc}
\usepackage{mathrsfs}
\usepackage{amssymb,amsfonts,amsmath}
\usepackage{amsthm}
\usepackage{enumerate}
\usepackage{times,fancyhdr}
\usepackage[dvips]{graphicx}
\usepackage{verbatim}

\makeatletter
\g@addto@macro{\endabstract}{\@setabstract}
\newcommand{\authorfootnotes}{\renewcommand\thefootnote{\@fnsymbol\c@footnote}}%
\makeatother

\numberwithin{equation}{section}

\newtheorem{theorem}{Theorem}
\newtheorem{proposition}{Proposition}
\newtheorem {lemma}{Lemma}

\newtheorem{corollary}[theorem]{Corollary}
\newtheorem{remark}{Remark}

\begin{document}

\begin{center}
  \LARGE
  Sufficient Stochastic Maximum Principle for Discounted Control Problem \par \bigskip

  \normalsize
  \authorfootnotes
  Bohdan Maslowski \footnote{The author was supported in part by the GACR grant No. P201/10/0752.}\textsuperscript{1} ,
  Petr Veverka \footnote{The author was supported by the Czech CTU grant SGS12/197/OHK4/3T/14, MSMT grant INGO II INFRA LG12020 and GACR GRANT No. P201/10/0752.}\textsuperscript{2} \par \bigskip
  \textsuperscript{1}Faculty of Mathematics and Physics, Charles University in Prague, Czech Republic, maslow@karlin.mff.cuni.cz \par

  \textsuperscript{2}Faculty of Nuclear Sciences and Physical Engineering, Czech Technical University in Prague, Czech Republic, petr.veverka@fjfi.cvut.cz \par \bigskip

\end{center}


\noindent \textbf{Abstract.} {\small In this article, the sufficient Pontryagin's maximum principle for infinite horizon discounted stochastic control problem is established. The sufficiency is ensured by an additional assumption of concavity of the Hamiltonian function.
Throughout the paper, it is assumed that the control domain $U$ is a convex set and the control may enter the diffusion term of the state equation. The general results are applied to the controlled stochastic logistic equation of population dynamics.



\section{\textbf{Introduction}}

\noindent In the present paper, the discounted stochastic optimal control problem is dealt with. This kind of problem is very popular and plentifully used in many domains, especially in stochastic finance since it leads to maximizing the average discounted agent's utility. The method of solving the problem used here is the maximum principle which, in deterministic setting, was formulated in 1950s by the group of L. S. Pontryagin. For diffusions, the maximum principle has been studied by many researchers. The earliest versions of a maximum principle for such process were suggested by Kushner \cite{kushner} and Bismut \cite{bismut}. Further progress on the subject was subsequently made by Bensoussan \cite{bensoussan}, Peng \cite{peng}, Cadenillas and Haussmann \cite{haussmann} and others.
Originally, the main technical tool used when considering maximum principle was the calculus of variations which was
not easy to apply to real examples and was not convenient for numerical simulations.  The turning point which led to its further intensive study was the paper by Pardoux and Peng \cite{pardouxpeng} where the general (nonlinear) problem of Backward Stochastic Differential Equation (BSDE in short) was formulated and the existence and uniqueness theorems were given. BSDE theory provides an elegant and easy-to-handle tool to describe the adjoint processes and to formulate the maximum principle by means of the Hamiltonian function. Also, numerical methods for this kind of processes are intensively studied. \\

\noindent These results were extended in numerous papers. For example, in case of diffusions with jumps, a necessary maximum principle on the finite time horizon was formulated by Tang and Li \cite{tang} whereas sufficient optimality conditions on finite time horizon were specified by \O ksendal, Sulem and Framstad \cite {oksendal}. Maximum principle on infinite time horizon was studied in Haadem, Proske and \O ksendal \cite{oksendal priklad}.
Interested reader can also find a large amount of papers on maximum principle for a variety of cases:
Singular control (Bahlali and Mezerdi \cite{bahlali mezerdi}, Dufour and Miller \cite{dufour}, \O ksendal and Sulem \cite{oksendal singular levy}), Impulse control (Chikodza \cite{chikodza}, Wu and Zhang \cite{wu zhang}),
Controlled SPDE's (Al-Hussein \cite{hussein}, Fuhrman, Hu and Tessitore \cite{fuhrman tessitore hu}),
Delayed controlled systems (Agram, Haadem, \O ksendal and Proske \cite{oks delay}, \O ksendal, Sulem and Zhang \cite{oks sulem zhang}), Near-optimal control (Zhou \cite{zhou near opt}) and many others. \\

\noindent Applying the Hamiltonian formalism to stochastic control problems, the class of Forward-Backward Stochastic Differential Equations (FBSDE in short) naturally arises in form of a partially-coupled system of the forward equation for the controlled diffusion and the adjoint backward equation for ,,generalized Lagrange multipliers''.
FBSDE with infinite (or random) time horizon are still subject to intensive study. The question which is quite delicate is the behaviour of the solution processes at infinity.
There are several papers (e.g. Pardoux \cite{pardoux}, Peng and Shi \cite{peng shi}, Yin \cite{yin}, Wu \cite{wu}) answering this question under different sets of assumptions both on the coefficients of the FBSDE and on the terminal condition. Some of these papers naturally require that the terminal condition of the backward equation at infinity (in certain sense) is given in advance with suitable properties that consequently determine the space for the solution processes.
Unfortunately, this is not directly applicable to stochastic control problems on infinite time horizon since  the terminal condition is usually not known. Therefore, a different approach has to be introduced, namely we  specify the solution space of random functions which determines asymptotical behaviour of the processes, suitable for our purpose.
This approach appears already in Peng and Shi \cite{peng shi} where however the conditions on coefficients, are rather restrictive -  global Lipschitz property in all variables is assumed (which, in fact, excludes polynomial coefficients of higher order than two) and some special kind of so called weak monotonicity (also known as
one-sided Lipschitz property).
In this case, the solution process is vanishing at infinity and therefore, the terminal
condition is zero a.s. In the paper by Wu \cite{wu}, a different monotonicity condition is assumed to
obtain a solution process with non zero (in general) yet still a.s. constant terminal condition. The most general result is due
to Yin \cite{yin} who weakens the assumptions to obtain the solution in some exponentially-weighted $\mathbf{L}^2$ space for some
suitable discount factor. Nevertheless, in the latter paper, existence of the solution to the backward part of the system employs the result by Pardoux \cite{pardoux} based on the knowledge of the terminal condition.\\

\noindent The novelty of present paper is twofold: First of all, the applicability of the maximum principle to infinite time horizon control problems has been studied under a natural set of conditions that are verifiable and may cover nonlinear
problems with non-Lipschitz coefficients. Secondly, these general results have been used to solve the corresponding control problem for controlled stochastic logistic equation (which indeed involves non-Lipschitz terms), which
is an important model of population dynamics. In this case also the solvability of the closed-loop FBSDE has been established.

\noindent The paper is organized as follows: The discounted control problem is formulated in the second section. In the third section the existence and uniqueness theorem of the solution to FBSDE on infinite time horizon is provided.
Fourth section contains the main result of the paper -  the sufficient infinite time maximum principle for the discounted problem.
In the fifth section, some standard examples are solved by using results of the fourth section to show the applicability.
The last section which contains our main example is devoted to the above mentioned controlled stochastic logistic equation.

\section{Formulation of the problem}
\subsection{Preliminaries}
We are given a basic probability space $\big(\Omega,\mathcal{F}, \mathbf{P} \big)$ with $\mathbb{R}^d$-valued standard Wiener process $W = \big(W_t \big)_{t\geq0}$. Let $\big(\mathcal{F}^W_t \big)_{t\geq0}$ be
the canonical filtration of $W$, i.e. $\mathcal{F}^W_t = \sigma\big(W_s; s \leq t \big)$, and
$\big(\mathcal{F}_t \big)_{t \geq 0}$ be its $\mathbf{P}-$null sets augmentation. We denote $\mathcal{F}_\infty = \bigvee_{t\geq0}\mathcal{F}_t \subset \mathcal{F}$. Further, to simplify the notation, we write just 'a.s.' instead of '$\mathbf{P}$ -a.s.'. We denote $|\cdot|$ and $||\cdot||$ the Euclidean norms in $\mathbb{R}^n$ and $\mathbb{R}^{n \times d}$ respectively. Further, $\big<\cdot,\cdot \big>$ stands for standard scalar product in $\mathbb{R}^n$ and $Tr(\cdot)$ denotes the trace of a square matrix.

\noindent Now, for $\beta \in \mathbb{R}$ and any Banach space $\mathcal{X}$ with norm $||\cdot||_{\mathcal{X}}$, we introduce the space of random processes
\begin{align}\label{eq:vazeny L2 prostor}\nonumber
\mathbf{L}^{2,\beta}_{\mathcal{F}}\left(\mathbb{R}_+;\mathcal{X} \right) := \left\{ v: \mathbb{R}_+ \times \Omega \rightarrow \mathcal{X}: v\ {\rm {is} }\ \left(\mathcal{F}_t\right)_{t\geq0}-\right. & \left. {\rm progressive\ with}\ \right. \\
&\left. \mathbf{E}\int^{+\infty}_0 e^{\beta t}||v_t||_{\mathcal{X}}^2dt < + \infty  \right\}.
\end{align}
\noindent We denote $\mathbf{L}^{2,0}_{\mathcal{F}}\left(\mathbb{R}_+;\mathcal{X} \right)$ just as $\mathbf{L}^2_{\mathcal{F}}\left(\mathbb{R}_+;\mathcal{X} \right)$ and

\begin{align}\label{eq:vazeny L2 prostor 2}\nonumber
\mathbf{L}^{2,\beta}_{\mathcal{F},loc}\left(\mathbb{R}_+;\mathcal{X} \right) := \left\{ v: \mathbb{R}_+ \times \Omega \rightarrow \mathcal{X}: v\ {\rm {is} }\ \left(\mathcal{F}_t\right)_{t\geq0}-\right. & \left. {\rm progressive\ with}\ \right. \\
&\left. \mathbf{E}\int^{T}_0 e^{\beta t}||v_t||_{\mathcal{X}}^2dt < + \infty,\ \forall T >0  \right\}.
\end{align}

\subsection{Discounted control problem}
The controlled state process $(X_t)_{t\geq 0}$ is a strong solution to the following controlled SDE in $\mathbb{R}^n$

\begin{align}\label{eq:difuse inf}
dX_t &= b(X_t,u_t)dt +  \sigma(X_t,u_t) dW_t, \quad \forall t\geq 0, \\ \nonumber
  X_0 &= x,
\end{align}

\noindent where $U$ is a convex subset of $\mathbb{R}^{k}$, the mappings $b:\mathbb{R}^n \times U \rightarrow \mathbb{R}^n$ and $\sigma:\mathbb{R}^n \times U \rightarrow \mathbb{R}^{n \times d}$ satisfy conditions specified in the next section.

\noindent Denote $\mathcal{U}_{ad}$ the set of all admissible controls as

\begin{equation}\label{eq:Uad}
\mathcal{U}_{ad} = \left\{ u: \mathbb{R}_+ \times \Omega \rightarrow U: u \in \mathbf{L}^2_{\mathcal{F}, loc}\left(0,+\infty;U \right)   \right\}.
\end{equation}

\noindent Any process $u(\cdot) \in \mathcal{U}_{ad}$ is called an admissible control. The cost functional takes the form

\begin{equation}\label{eq:funkcional}
J(u(\cdot)) = \mathbf{E} \int^{+\infty}_0 e^{-\beta t}f(X_t,u_t)dt,
\end{equation}

\noindent where $f: \mathbb{R}^n \times U \rightarrow \mathbb{R}$ is the penalization (or appreciation) function such that $J(u(\cdot))$ converges for every admissible control and $\beta >0$ is the discount factor. \\
Further, define the value function $v$ by

\begin{equation}\label{eq:cost fce}
v = \operatorname*{sup}_{u(\cdot) \in \mathcal{U}_{ad}}J(u(\cdot)).
\end{equation}

\noindent The goal is to find such a strategy $u^*(\cdot) \in \mathcal{U}_{ad}$ so that  the supremum in (\ref{eq:cost fce}) is attained in $u^*(\cdot)$, i.e. $v = J(u^*(\cdot))$.
For the discounted control problem (\ref{eq:difuse inf}), (\ref{eq:Uad}), (\ref{eq:funkcional}) and (\ref{eq:cost fce}) we use the abbreviation DCP.

\noindent Define a generalized Hamiltonian function $\mathcal{H} $ associated to control problem (\ref{eq:difuse inf}) - (\ref{eq:cost fce}) by $\mathcal{H}: \mathbb{R}^n\times U\times \mathbb{R}^n\times\mathbb{R}^{n\times d} \rightarrow \mathbb{R}$ as

\begin{equation}\label{eq:hamiltonian}
\mathcal{H}(x,u,y,z) = \big< b(x,u),y\big> + Tr(\sigma(x,u)'z) + f(x,u) -\beta \big<x,y\big>.
\end{equation}

\noindent The Hamiltonian is an analogue to the Lagrange function in the theory of constrained optimization since the variables $y$ and $z$ can be viewed as 'generalized Lagrange multipliers' and the functions $b$ and $\sigma$ as the constraints for the dynamics of the space process $X_t$. The additional term $-\beta \big<x,y\big>$ comes up from the
Lyapunov function of the FBSDE system, see Peng, Shi \cite{peng shi}.\\
\noindent We note that neither the extremal point of $\mathcal{H}$ w.r.t. $u$ nor the concavity/convexity w.r.t. $(x,u)$ do not depend on the last term $-\beta \big<x,y\big>$ .

\noindent Suppose further that $\mathcal{H}$ is differentiable in $x$ (with the gradient denoted as $\nabla_x\mathcal{H}$) and we consider the following BSDE

\begin{equation}\label{eq:hamilt bsde inf}
-dY_t = \nabla_x\mathcal{H}(X_t,u_t,Y_t,Z_t )dt - Z_tdW_t, \quad \forall t \geq 0,
\end{equation}
which, in application to the DCP, reads

\begin{equation}\label{eq:bsde inf}\nonumber
-dY_t =  \Big( \nabla_x b(X_t,u_t) Y_t + D_x \sigma(X_t,u_t) \cdot Z_t + \nabla_x f(X_t,u_t) -\beta Y_t \Big) dt - Z_t dW_t, \quad \forall t \geq 0.
\end{equation}

\noindent We use the notation $D_x \sigma(x,u) \cdot z = \sum_{i=1}^d \nabla_x \sigma^i(x,u)z^i \in \mathbb{R}^n,$ for $z \in \mathbb{R}^{n \times d}$, where $\sigma^i$ denotes the $i$-th column of the matrix $\sigma$.

\section{Solution to controlled decoupled FBSDE on infinite time horizon}

\noindent In this section, we study the following controlled FBSDE associated to the DCP

\begin{align}\label{eq:FBSDE}
dX_t &= b(X_t,u_t)dt + \sigma(X_t,u_t) dW_t,\ \forall t \geq 0 \\ \nonumber
  X_0 &= x \in \mathbb{R}^n, \\ \nonumber
-dY_t &=  \Big( \nabla_x b(X_t,u_t) Y_t + D_x \sigma(X_t,u_t) \cdot Z_t + \nabla_x f(X_t,u_t) -\beta Y_t \Big) dt - Z_t dW_t,\ \forall t \geq 0.
\end{align}

\noindent Since the system (\ref{eq:FBSDE}) is partially-coupled, i.e. the forward part does not depend on the solution of the backward equation, we can solve the forward part separately on
$\mathbb{R}_+$ obtaining the process $X$. In the next step, plugging the solution $X$ into the backward part, we find the solution processes $(Y,Z)$.
We start with results on infinite time horizon SDE which is given by

\begin{align}\label{eq:difuse inf new}
dX_t &= b(X_t,u_t)dt +  \sigma(X_t,u_t) dW_t, \quad \forall t\geq 0, \\ \nonumber
  X_0 &= x,
\end{align}

\noindent where $U$ is a convex bounded subset of $\mathbb{R}^{k}$ and $x \in \mathbb{R}^n$ is a deterministic initial condition. We further assume that the mappings

\begin{align}\nonumber
& b: \mathbb{R}^n \times U \rightarrow \mathbb{R}^n, \nonumber \\
& \sigma: \mathbb{R}^n \times U \rightarrow \mathbb{R}^{n \times d},\nonumber
\end{align}
satisfy the following conditions.

\begin{description}

\item[(H1)] $b(\cdot,\cdot)$ and $\sigma (\cdot,\cdot)$ are continuous on $\mathbb{R}^n \times \overline{U}$. \\

\item[(H2)] There exists $\mu_1 \in \mathbb{R}$ such that
\vspace{-2mm}
\begin{align}\label{eq:SDE H2}
\langle x_1-x_2,  b(x_1,u) - b(x_2,u)\rangle &\leq \mu_1 |x_1 - x_2|^2,
\end{align}
for any $(x_1, x_2, u) \in \mathbb{R}^n \times \mathbb{R}^n \times U$.

\item[(H3)](Lipschitz property) There exists a constant $L > 0$ such that
\vspace{-2mm}
\begin{equation}\label{eq:SDE H3}
|| \sigma(x_1,u) - \sigma(x_2,u)|| \leq L |x_1 - x_2|,
\end{equation}
for any $(x_1, x_2, u) \in \mathbb{R}^n \times \mathbb{R}^n \times U$.
\end{description}

\noindent The following choice of the solution space will allow to apply Theorem 1 bellow to sufficient
maximum principle. As we will see, the suitable solution space ensuring proper asymptotics is the space $\mathbf{L}^{2,-\beta}_{\mathcal{F}}\left(\mathbb{R}_+;\mathbb{R}^n \right)$ for some $\beta > 0$.
The following theorem holds.

\begin{theorem}\label{existence dopredne}[Existence and uniqueness] Let assumptions (H1)-(H3) hold.
Then for each control $u(\cdot) \in \mathcal{U}_{ad}$ and every initial condition $x \in \mathbb{R}^n$, SDE (\ref{eq:difuse inf new}) admits a unique solution $X = X^{x,u(\cdot)} \in \mathbf{L}^{2,-\beta}_{\mathcal{F}}\left(\mathbb{R}_+;\mathbb{R}^n \right)$ where $\beta - 2\mu_1 - 2L^2 >   0$.
\end{theorem}

\begin{proof}[\bf{Proof:}]
Uniqueness follows from Lemma \ref{SDE apriori} below.
The proof of existence is similar to Theorem 2.2. in Friedman \cite{friedman} and it will be only sketched.
First note that by virtue of (H1)-(H3) we have that for all $\varepsilon > 0$ there are constants $K_{\varepsilon}, \overline{K}$
($\overline{K}$ independent of $\varepsilon$) such that

\begin{align}\label{eq:SDE odhady}
\langle x, b(x,u)\rangle &\leq \left(\mu_1 + \varepsilon \right) |x|^2 + K_{\varepsilon}, \nonumber \\
|| \sigma(x,u)|| &\leq L |x| + \overline{K},
\end{align}
for all $x>0, u \in U$. Now for all $n \in \mathbb{N}$ define

\begin{equation}\label{eq:SDE markov cas}
\tau_n = \operatorname*{inf}\{t > 0: \left|X^n_t \right| \geq n \},
\end{equation}
where $(X^n_t)$ is a local solution obtained for trimmed coefficients on a ball of radius $n$ and for time $t \in [0,\tau_n]$. Due to (H1)-(H3) and \ref{eq:SDE odhady} one obtains

\begin{align}\label{eq:SDE norma reseni}
\operatorname*{sup}_{t \geq 0} \mathbf{E}&\left[e^{-\beta (t\wedge \tau_n)}\left|X^n_{t\wedge \tau_n}\right|^2 \right]
+ \left(\beta -2(\mu_1 + \varepsilon) - 2L^2 \right)\mathbf{E} \int^{\tau_n}_0 e^{-\beta (s\wedge \tau_n)}\left|X^n_{s\wedge \tau_n}\right|^2 ds  \nonumber \\
&\leq |x_0|^2 + \frac{2}{\beta}\left(K_{\varepsilon} + \overline{K}^2 \right) < + \infty.
\end{align}

\noindent A similar estimate can be easily obtained with the supremum inside the expectation. By using the standard consistency argument and Lyapunov techniques with Lyapunov function $V(x) = 1 + |x|^2$ we conclude that there is a limit process $(X_t)_{t \geq 0}$ in $\mathbf{L}^{2,-\beta}_{\mathcal{F}}\left(\mathbb{R}_+;\mathbb{R}^n \right)$ with $\beta - 2\mu_1 - 2L^2 > 0$.
\end{proof}

\begin{lemma}\label{SDE apriori}[A priori estimate for SDE] Let assumptions (H1)-(H3) hold and $u(\cdot) \in \mathcal{U}_{ad}$ be arbitrary but fixed. Let $X^1 = X^{u(\cdot),1}$ and $X^2= X^{u(\cdot),2}$ be two solutions to (\ref{eq:difuse inf new}) with initial values $x_1$ and $x_2$ respectively, belonging to $\mathbf{L}^{2,-\beta}_{\mathcal{F}}\left(\mathbb{R}_+;\mathbb{R}^n \right)$. Then the following a priori estimate holds
\begin{align}\label{eq:SDE apriori}
\operatorname*{sup}_{t \geq 0} \mathbf{E}\left( e^{-\beta t} |\widehat{X}_t|^2 \right)+ \left(\beta -2\mu_1 - 2L^2 \right) \mathbf{E}\int^{+\infty}_0 e^{-\beta t} |\widehat{X}_s|^2 ds \leq
|x_1 - x_2|^2,
\end{align}
for $\beta >  2\mu_1 + 2L^2$ and $\widehat{X}_t = X^1_t - X^2_t$.
\end{lemma}

\begin{proof}[\bf{Proof:}]
First assume that $\beta \in \mathbb{R}$ is arbitrary and fix some $t>0$. At the end of the proof, the right value of $\beta$ will be specified.
Using It\^o formula to $e^{-\beta t}|\widehat{X}_t|^2$ on $[0,t]$ one arrives at
\begin{align}\label{eq:SDE apriori Ito}
e^{-\beta t}|\widehat{X}_t|^2 &+ \beta \int^{t}_0 e^{-\beta s} |\widehat{X}_s|^2 ds \nonumber \\
&=|x_1 - x_2|^2 + 2\int^{t}_0 e^{-\beta s} \big<\widehat{X}_s, b(X^1_s,u_s)- b(X^2_s,u_s) \big>ds  \nonumber \\
&+2\int^{t}_0 e^{-\beta s} \big<\widehat{X}_s, \left(\sigma(X^1_s,u_s)- \sigma(X^2_s,u_s)\right)dW_s \big> \nonumber \\
&+ \int^{t}_0 e^{-\beta s} ||\sigma(X^1_s,u_s)- \sigma(X^2_s,u_s)||^2 ds.
\end{align}

\noindent Now, taking $\mathbf{E}(\cdot)$ and employing (H2)-(H3) we obtain

\begin{align}\label{eq:SDE apriori Ito2}
\mathbf{E} \left( e^{-\beta t}|\widehat{X}_t|^2 \right) &+ \left(\beta - 2\mu_1 - 2L^2 \right) \mathbf{E}\int^{t}_0 e^{-\beta s} |\widehat{X}_s|^2 ds  \leq |x_1 - x_2|^2.
\end{align}
Then for $\beta >  2\mu_1 + 2L^2$, we get (\ref{eq:SDE apriori}).
\end{proof}

\vspace{3mm}

\begin{remark} One could easily obtain a similar estimate to (\ref{eq:SDE apriori}) but with the supremum inside the expectation. In that case, we have to consider $\beta >  2\mu_1 + 2L^2(C^2 + 1)$ where
$C$ is the constant from Burkholder-Davis-Gundy inequality. For purposes of deriving the sufficient maximum principle, the estimation (\ref{eq:SDE apriori}) is satisfactory.
\end{remark}

Now, we will be interested in the solution to the BSDE (\ref{eq:FBSDE}).
We assume that we already know the solution $X = X^{u(\cdot)} \in \mathbf{L}^{2,-\beta}_{\mathcal{F}}\left(\mathbb{R}_+;\mathbb{R}^n \right)$ to SDE
(\ref{eq:difuse inf new}) for some fixed $u(\cdot) \in \mathcal{U}_{ad}$ assuming (H1)-(H3). The couple $(X,u)$ will be the (fixed) input for the BSDE (\ref{eq:FBSDE}).

\noindent We start the study of the backward equation by further specifying the coefficients $b,\sigma,f$. We assume that

\begin{align}
& b: \mathbb{R}^n \times U \rightarrow \mathbb{R}^n, \nonumber \\
& \sigma: \mathbb{R}^n \times U \rightarrow \mathbb{R}^{n \times d}, \nonumber \\
& f: \mathbb{R}^n \times U \times \rightarrow \mathbb{R},\nonumber
\end{align}
such that
\begin{description}

\item[(H4)] $b(x,u), \sigma(x,u)$ and $f(x,u)$ are continuously differentiable in $x$ for all
$u \in U$. \\

\item[(H5)] There exists $\mu_2 \in \mathbb{R}$ such that
\vspace{-2mm}
\begin{align}\label{eq:BSDE H5}
\langle y_1-y_2, \nabla_x b(x,u)(y_1-y_2 )\rangle &\leq \mu_2 |y_1 - y_2|^2,
\end{align}
for any $(x, u, y_1, y_2) \in \mathbb{R}^n \times U \times \mathbb{R}^n \times \mathbb{R}^n$.

\item[(H6)] There exists a constant $M \geq 0$ such that
\vspace{-2mm}
\begin{equation}\label{eq:BSDE H6}
|| D_x \sigma(x,u)|| := \sum^d_{i=1}||\nabla_x\sigma^i(x,u)||  \leq M,
\end{equation}
for any $(x, u) \in \mathbb{R}^n \times U$.
\end{description}

\noindent Using the transform $Y_t = e^{\beta t}\widetilde{Y}_t$ and $Z_t = e^{\beta t}\widetilde{Z}_t$, the infinite horizon BSDE (\ref{eq:FBSDE}) can be rewritten as

\begin{equation}\label{eq:bsde inf tilde}
-d\widetilde{Y}_t =  \Big( \nabla_x b(X_t,u_t) \widetilde{Y}_t + D_x \sigma(X_t,u_t) \cdot \widetilde{Z}_t + e^{-\beta t}\nabla_x f(X_t,u_t) \Big) dt - \widetilde{Z}_t dW_t, \quad \forall t \geq 0,
\end{equation}

\noindent If we show that the solution process $(\widetilde{Y}, \widetilde{Z})$ is in $\mathbf{L}^{2,\beta}_{\mathcal{F}}\left(\mathbb{R}_+;\mathbb{R}^n \times \mathbb{R}^{n \times d}\right)$ then necessarily the original solution process $(Y,Z)$ has to be in $\mathbf{L}^{2,-\beta}_{\mathcal{F}}\left(\mathbb{R}_+;\mathbb{R}^n \times \mathbb{R}^{n \times d}\right)$ since we have

\begin{equation}\label{eq:l2 weighted space}
\mathbf{E}\int^{+\infty}_0 e^{-\beta t} |Y_t|^2dt =  \mathbf{E}\int^{+\infty}_0 e^{-\beta t} |e^{\beta t}\widetilde{Y}_t|^2dt = \mathbf{E}\int^{+\infty}_0 e^{\beta t} |\widetilde{Y}_t|^2dt < + \infty,
\end{equation}
and similarly for the process $Z$.

\vspace{3mm}

\noindent The existence and uniqueness is given by Theorem 4 in Peng and Shi \cite{peng shi}. In our setting, it is formulated as follows:

\begin{theorem}\label{thm:reseni BSDE}[Existence and uniqueness] Let assumptions (H4)-(H6) hold and $\left(X^{u(\cdot)}_t, u(\cdot) \right)$ be an admissible
couple such that $X^{u(\cdot)}_t \in \mathbf{L}^{2,-\beta}_{\mathcal{F}}\left(\mathbb{R}_+;\mathbb{R}^n \right)$ and
$\nabla_x f(X_t,u_t) \in \mathbf{L}^{2,-\beta}_{\mathcal{F}}\left(\mathbb{R}_+;\mathbb{R}^n \right)$ where $\beta > 2\mu_2 + 2M^2$.
Then BSDE (\ref{eq:bsde inf tilde}) admits a unique solution $(\widetilde{Y}, \widetilde{Z}) \in \mathbf{L}^{2,\beta}_{\mathcal{F}}\left(\mathbb{R}_+;\mathbb{R}^n \times \mathbb{R}^{n \times d}\right)$.
\end{theorem}

\begin{remark}\label{rm:znamenko bety} The sign of $\beta$ depends on the sign of the factor $2\mu_2 + 2M^2$. In case of $\mu_2 <0$ one gets so called weak monotonicity condition as in Peng, Shi \cite{peng shi} allowing
also $\beta \leq 0$. As we will see later in Example 3, the controlled logistic equation naturally leads to $\mu_2 >0$ which implies $\beta >0.$
\end{remark}

\noindent Using the previous theorem and (\ref{eq:l2 weighted space}) we conclude this section by

\begin{corollary}\label{cor:reseni fbsde} Let assumptions (H1)-(H6) hold and suppose that $\nabla_x f(X_t,u_t) \in \mathbf{L}^{2,-\beta}_{\mathcal{F}}\left(\mathbb{R}_+;\mathbb{R}^n \right)$.
Then BSDE (\ref{eq:FBSDE}) admits a unique solution
$(Y, Z) \in \mathbf{L}^{2,-\beta}_{\mathcal{F}}\left(\mathbb{R}_+;\mathbb{R}^n \times \mathbb{R}^{n \times d}\right)$ for $\beta >2\mu_2 + 2M^2$.
\end{corollary}

\begin{remark}\label{rm:oblast G} The above theorems stand valid if the assumptions (H1)-(H6) are fulfilled for all $x \in G \subset \mathbb{R}^n$, $G$ domain, provided $X_t \in G$ for all $t \geq 0$ a.s.
\end{remark}

\subsection{Stability of BSDE under approximations}
In this subsection we legitimate the approximation the BSDE by an equation with terminal zero condition.
As we have already mentioned when dealing with DCP, the terminal value $\xi$ of the solution process
$Y$ is not known. We will show that under (H1)-(H6) and natural assumptions on $\xi$ (finite second moment)
the approximation is stable, i.e. one can choose zero terminal condition or $\mathcal{F}_t$ - measurable its projections to obtain the same solution process. The following lemma connects the approach in
Darling and Pardoux \cite{darling pardoux} and in Peng and Shi \cite{peng shi}.

\begin{lemma}\label{l:stabilita aproximaci}[Stability under approximations]
Let (H1)-(H6) hold and let $\xi$ 
be an $\mathbb{R}^n$-valued $\mathcal{F}_{\infty}$ - measurable random variable with $\mathbf{E}|\xi|^2< +\infty.$
Further, for each $n \in \mathbb{N}$, let us consider two BSDE's

\begin{align}\label{eq:lemma bsde1}
-d\widetilde{Y}^n_t &= \nabla_x \mathcal{H}(X_t,u_t,\widetilde{Y}^n_t,\widetilde{Z}^n_t)dt - \widetilde{Z}^n_tdW_t,\ t \in [0,n) \nonumber\\
\widetilde{Y}^n_n &= 0,
\end{align}
\noindent and
\begin{align}\label{eq:lemma bsde2}
-d\widehat{Y}^n_t &= \nabla_x \mathcal{H}(X_t,u_t,\widehat{Y}^n_t,\widehat{Z}^n_t)dt - \widehat{Z}^n_tdW_t,\ t \in [0,n) \nonumber\\
\widehat{Y}^n_n &= \xi_n := \mathbf{E}\left[\xi|\mathcal{F}_n\right],
\end{align}

\noindent where $\nabla_x \mathcal{H}(x,u,y,z) = \nabla_x b(x,u) y + D_x \sigma(x,u) \cdot z + \nabla_x f(x,u) -\beta y.$ \\
\noindent We lay $\widetilde{Y}^n_t = 0$ and $\widehat{Y}^n_t = \xi_t := \mathbf{E}\left[\xi|\mathcal{F}_t\right], \forall t>n$. \\
\noindent Then

\begin{align}\label{eq:aproximace}
\operatorname*{lim}_{n \rightarrow +\infty}
\left[ \operatorname*{sup}_{t \geq 0} \mathbf{E} \left(e^{-\beta t}|\widetilde{Y}^n_t - \widehat{Y}^n_t|^2\right) \right.&+
\mathbf{E} \int^{+\infty}_0 \mathbf{1}_{[0,n]}(t)e^{-\beta t}|\widetilde{Y}^n_t - \widehat{Y}^n_t|^2 dt \nonumber \\
&\left.+\mathbf{E} \int^{+\infty}_0 \mathbf{1}_{[0,n]} (t) e^{-\beta t}||\widetilde{Z}^n_t - \widehat{Z}^n_t||^2 dt \right] = 0,
\end{align}

for $\beta \geq 2\mu_2 + 2M^2 >0$.
\end{lemma}

\begin{proof}[\bf{Proof:}]
\noindent First observe that, using Jensen's inequality and integrability of $\xi$, one gets

\begin{equation}\label{eq:terminal approx}
\mathbf{E}\left(e^{-\beta n}|\xi_n|^2 \right) = e^{-\beta n} \mathbf{E}\left(|\mathbf{E}\left(\xi|\mathcal{F}_n\right)|^2 \right) \leq
\mathbf{E}\left(e^{-\beta n}|\xi|^2 \right) \operatorname*{\rightarrow}_{n \rightarrow +\infty}0.
\end{equation}

\noindent Now, applying It\^o formula to $e^{-\beta t}|\widetilde{Y}^n_t - \widehat{Y}^n_t|^2$ one arrives at

\begin{align}\label{eq:lemma aprox ito1}
\mathbf{E}&\left(e^{-\beta t}|\widetilde{Y}^n_t - \widehat{Y}^n_t|^2 \right) + \beta \mathbf{E} \int^{n}_0 e^{-\beta s} |\widetilde{Y}^n_s - \widehat{Y}^n_s|^2ds + \mathbf{E} \int^{n}_0 e^{-\beta s} ||\widetilde{Z}^n_s - \widehat{Z}^n_s||^2ds= \mathbf{E}\left(e^{-\beta n}|\xi_n|^2 \right)
\nonumber \\
&+ 2\mathbf{E} \int^{n}_0 e^{-\beta s}\Big[ \big<\nabla_x b(X_s,u_s)(\widetilde{Y}^n_s - \widehat{Y}^n_s),\widetilde{Y}^n_s - \widehat{Y}^n_s\big>
+ \big<D_x \sigma(X_s,u_s)\cdot(\widetilde{Z}^n_s - \widehat{Z}^n_s),\widetilde{Y}^n_s - \widehat{Y}^n_s\big>  \Big]ds.
\end{align}
\noindent Now, due to (H5) and (H6), we arrive at

\begin{align}\label{eq:lemma aprox ito 2}
\mathbf{E}\left(e^{-\beta t}|\widetilde{Y}^n_t - \widehat{Y}^n_t|^2 \right) &+ (\beta - 2\mu_2 - 2M^2)\mathbf{E} \int^{n}_0 e^{-\beta s}|\widetilde{Y}^n_s - \widehat{Y}^n_s|^2 ds  \nonumber \\
&+ \frac{1}{2}\mathbf{E} \int^{n}_0 e^{-\beta s}||\widetilde{Z}^n_s - \widehat{Z}^n_s||^2 ds \leq  \mathbf{E}\left(e^{-\beta n}|\xi_n|^2 \right),\ \forall 0 \leq t \leq n.
\end{align}

\noindent Taking the $ \operatorname*{sup}_{0 \leq t \leq n}$ on both sides of (\ref{eq:lemma aprox ito 2}), we obtain

\begin{align}\label{eq:lemma aprox ito 3}
\operatorname*{sup}_{0 \leq t \leq n} \mathbf{E}\left(e^{-\beta t}|\widetilde{Y}^n_t - \widehat{Y}^n_t|^2 \right) &+ (\beta - 2\mu_2 - 2M^2)\mathbf{E} \int^{n}_0 e^{-\beta s}|\widetilde{Y}^n_s - \widehat{Y}^n_s|^2 ds  \nonumber \\ &+ \frac{1}{2}\mathbf{E} \int^{n}_0 e^{-\beta s}||\widetilde{Z}^n_s - \widehat{Z}^n_s||^2 ds \leq  \mathbf{E}\left(e^{-\beta n}|\xi_n|^2 \right) \operatorname*{\rightarrow}_{n \rightarrow +\infty}0.
\end{align}

\noindent To obtain the estimate for all $t \geq 0$ and for $\beta \geq 2\mu_2 + 2M^2$ we write

\begin{align}\label{eq:lemma aprox ito 4}
\operatorname*{sup}_{t \geq 0} \mathbf{E}&\left(e^{-\beta t}|\widetilde{Y}^n_t - \widehat{Y}^n_t|^2 \right) \leq \operatorname*{sup}_{0 \leq t \leq n} \mathbf{E}\left(e^{-\beta t}|\widetilde{Y}^n_t - \widehat{Y}^n_t|^2 \right)  + \operatorname*{sup}_{t > n} \mathbf{E}\left(e^{-\beta t}|\widetilde{Y}^n_t - \widehat{Y}^n_t|^2 \right) \nonumber \\
& \leq \mathbf{E}\left(e^{-\beta n}|\xi_n|^2 \right) +  \operatorname*{sup}_{t > n} \mathbf{E}\left(e^{-\beta t}|\xi_t|^2 \right) \operatorname*{\rightarrow}_{n \rightarrow +\infty}0.
\end{align}

\noindent The last term is due to definition of $\widetilde{Y}^n$ and $\widehat{Y}^n$ for $t \geq 0$ and due to (\ref{eq:terminal approx}). The convergence of the integral terms in
(\ref{eq:aproximace}) follows easily.
\end{proof}

\begin{remark} In the above Lemma \ref{l:stabilita aproximaci}, we could impose even more general conditions on $\xi$ ensuring $\mathbf{E}\left(e^{-\beta n}|\xi_n|^2 \right) \rightarrow 0$ as $ n \rightarrow +\infty.$
\end{remark}

\section{Sufficient stochastic maximum principle for the DCP}
In this section, we return to the discounted control problem and prove the corresponding sufficient maximum
principle. First, we employ the approach using the results from the previous section. Second, we apply results from
\O ksendal et al. \cite{oksendal priklad} to our definition of Hamiltonian and prove the associated DCP using
transversality condition. Finally, the connection between the two methods is shown.

\begin{theorem}\label{thm:postac PM1}[Sufficient stochastic maximum principle]\label{theorem SMP asympt}
{\it Let (H1) - (H6) hold and $\beta > \operatorname{max}\left\{2\mu_1 + 2L^2, 2\mu_2 + 2M^2 \right\}$.
Assume moreover that $\nabla_x f(X_t,u_t) \in \mathbf{L}^{2,-\beta}_{\mathcal{F}}\left(\mathbb{R}_+;\mathbb{R}^n \right)$ for any admissible couple
$(X,u)$ from Theorem \ref{existence dopredne}.
Further, let $\widehat{u}(\cdot) \in \mathcal{U}_{ad}$ and $\widehat{X}$ be the associated controlled diffusion process. Let us suppose that there exists a solution $(\widehat{Y},\widehat{Z})$ to the associated BSDE (\ref{eq:FBSDE}) such that

\begin{enumerate}
\item $\mathcal{H}(\widehat{X}_t,\widehat{u}_t,\widehat{Y}_t,\widehat{Z}_t) = \mathop{max}\limits_{u \in U} \mathcal{H}(\widehat{X}_t,u,\widehat{Y}_t,\widehat{Z}_t),
\ \mathbf{P}\otimes dt -{\rm a.e.},$

\item $(x,u) \rightarrow \mathcal{H}(x,u,\widehat{Y}_t,\widehat{Z}_t)$ is a concave function,\ $\mathbf{P}\otimes dt -{\rm a.e.}$
\end{enumerate}

\vspace{3mm}

Then $\widehat{u}(\cdot) = u^*(\cdot)$, i.e. $\widehat{u}(\cdot)$ is the optimal control strategy to the DCP.}
\end{theorem}

\begin{proof}[\bf{Proof:}] Take an arbitrary $u(\cdot) \in \mathcal{U}_{ad}$ and examine the difference $J(\widehat{u}(\cdot)) - J(u(\cdot))$. The goal is to show that this quantity is nonnegative.
Using the definition of $J(u(\cdot))$ and $\mathcal{H}$ we have
\begin{align}\label{eq:fce upravy}
J(&\widehat{u}(\cdot)) - J(u(\cdot)) = \mathbf{E} \int^{+\infty}_0 e^{-\beta t}\left(f(\widehat{X}_t,\widehat{u}_t) - f(X_t,u_t)\right)dt  \nonumber \\
&=\mathbf{E} \int^{+\infty}_0 e^{-\beta t}\left(\mathcal{H}(\widehat{X}_t,\widehat{u}_t,\widehat{Y}_t,\widehat{Z}_t) - \mathcal{H}(X_t,u_t,\widehat{Y}_t,\widehat{Z}_t) + \big< b(X_t,u_t)-b(\widehat{X}_t,\widehat{u}_t),\widehat{Y}_t \big> \right. \nonumber \\
&\left.\quad + Tr\left\{\big(\sigma'(X_t,u_t)-\sigma'(\widehat{X}_t,\widehat{u}_t)\big)\widehat{Z}_t\right\} + \beta \big<\widehat{X}_t - X_t ,\widehat{Y}_t\big> \right)dt .
\end{align}

\noindent Clearly, the r.h.s. of (\ref{eq:fce upravy}) is finite hence we obtain

\begin{equation}\label{eq:funkcional limita 1}
\mathbf{E} \int^{+\infty}_0 \mathcal{I}_t dt = \lim_{T \to +\infty} \mathbf{E} \int^{T}_0 \mathcal{I}_t dt,
\end{equation}

\noindent where $\mathcal{I}_t$ is the integrand of (\ref{eq:fce upravy}). \\
Further, since (H1)-(H6) hold we know by Corollary \ref{cor:reseni fbsde} that each of the three solution processes $X,\widehat{X},\widehat{Y}$ belongs to $\mathbf{L}^{2,-\beta}_{\mathcal{F}}\left(\mathbb{R}_+;\mathbb{R}^n \right)$ and therefore there exists a sequence $\left(T_n\right)_{n \in \mathbb{N}},\ T_n \nearrow +\infty$ as $n \rightarrow +\infty$ so that

\begin{equation}\label{eq:FBSDE a asymptotika}
\left|\mathbf{E} \left( e^{-\beta T_n}\big<\widehat{X}_{T_n} - X_{T_n},\widehat{Y}_{T_n}\big>\right) \right| \leq
\frac{1}{2}\mathbf{E}\left( e^{-\beta {T_n}}|\widehat{X}_{T_n} - X_{T_n}|^2 \right) + \frac{1}{2} \mathbf{E}\left( e^{-\beta {T_n}}|\widehat{Y}_{T_n}|^2\right) \operatorname*{\rightarrow}_{n \rightarrow +\infty}0,
\end{equation}

\noindent where we have applied Cauchy-Schwarz inequality and the Young inequality $2ab \leq \frac{1}{2}a^2 + \frac{1}{2}b^2$, $a,b \in \mathbb{R}$.
\noindent By (\ref{eq:funkcional limita 1}) and (\ref{eq:FBSDE a asymptotika}) we have that

\begin{equation}\label{eq:funkcional limita 2}
J(\widehat{u}(\cdot)) - J(u(\cdot)) = \lim_{n \to +\infty} \mathbf{E}\left[ \int^{T_n}_0 \mathcal{I}_t dt + e^{-\beta T_n}\big<\widehat{X}_{T_n} - X_{T_n},\widehat{Y}_{T_n}\big> \right].
\end{equation}

\noindent Now, applying It\^o formula to the last term in the bracket and taking $\mathbf{E}(\cdot)$ we arrive to

\begin{align}\label{eq:Ito na smiseny clen}
\mathbf{E}&\left[ e^{-\beta T_n}\big<\widehat{X}_{T_n} - X_{T_n},\widehat{Y}_{T_n}\big> \right] = \mathbf{E} \int^{T_n}_0 e^{-\beta t} \left( \big< b(\widehat{X}_t,\widehat{u}_t) - b(X_t,u_t),\widehat{Y}_t \big> \right. \nonumber \\
&\left.\quad + Tr\left\{\big(\sigma'(\widehat{X}_t,\widehat{u}_t) - \sigma'(X_t,u_t)\big)\widehat{Z}_t\right\} - \beta \big<\widehat{X}_t - X_t ,\widehat{Y}_t\big> \right. \nonumber \\
&\left.\quad \quad \quad \quad \quad \quad \quad \quad \quad \quad \quad \quad -\big<\widehat{X}_t - X_t, \nabla_x\mathcal{H}(\widehat{X}_t,\widehat{u}_t,\widehat{Y}_t,\widehat{Z}_t )\big>\right)dt.
\end{align}

\noindent In view of the equality (\ref{eq:funkcional limita 2}), we finally arrive at

\begin{align}\label{eq:SPM end}
J(\widehat{u}(\cdot)) - J(u(\cdot))= \operatorname*{lim}_{n \rightarrow +\infty}\mathbf{E} \int^{T_n}_0
&e^{-\beta t} \left(\mathcal{H}(\widehat{X}_t,\widehat{u}_t,\widehat{Y}_t,\widehat{Z}_t) - \mathcal{H}(X_t,u_t,\widehat{Y}_t,\widehat{Z}_t) \nonumber \right. \\
& \left. -\big<\widehat{X}_t - X_t, \nabla_x\mathcal{H}(\widehat{X}_t,\widehat{u}_t,\widehat{Y}_t,\widehat{Z}_t )\big> \right)dt.
\end{align}

\noindent By the concavity of $\mathcal{H}$ in $(x,u)$, we have that

\begin{equation}\label{eq:1. clen 1}
\mathcal{H}(\widehat{X}_t,\widehat{u}_t,\widehat{Y}_t,\widehat{Z}_t) - \mathcal{H}(X_t,u_t,\widehat{Y}_t,\widehat{Z}_t) -
\big<\widehat{X}_t - X_t, \nabla_x\mathcal{H}(\widehat{X}_t,\widehat{u}_t,\widehat{Y}_t,\widehat{Z}_t )\big> \geq 0.
\end{equation}

\noindent Therefore, we deduce that

$$J(\widehat{u}(\cdot)) - J(u(\cdot)) \geq 0, \ \ \forall u(\cdot) \in \mathcal{U}_{ad},$$

\noindent which proves that $\widehat{u}(\cdot)$ is indeed the optimal control.
\end{proof}

\noindent Now, we provide a similar version of sufficient maximum principle proved in \cite{oksendal priklad}, Theorem 4.1 using the so called transversality condition (TVC).
The fact which may be inconvenient for the controller is that the TVC has to be verified (theoretically)
for every admissible control $u(\cdot)$. Moreover, this criterion, in general, cannot be explicitly verified in terms of the
coefficients $b, \sigma, f$ and $\beta$ of the DCP. On the other hand, TVC has an economical interpretation for the DCP, namely it expresses
the additional benefit of one unit of good for time increasing to infinity.

\begin{theorem}\label{postac PM2}[Sufficient stochastic maximum principle]
 Let $\widehat{u}(\cdot) \in \mathcal{U}_{ad}$ and $\widehat{X}$ be the associated controlled diffusion process. Let us suppose that there exists a solution $(\widehat{Y},\widehat{Z})$ to the associated BSDE (\ref{eq:FBSDE}) such that

\begin{enumerate}
\item $\mathcal{H}(\widehat{X}_t,\widehat{u}_t,\widehat{Y}_t,\widehat{Z}_t) = \mathop{max}\limits_{u \in U} \mathcal{H}(\widehat{X}_t,u,\widehat{Y}_t,\widehat{Z}_t),
\quad \mathbf{P}\otimes dt -{\rm a.e.},$

\item $(x,u) \rightarrow \mathcal{H}(x,u,\widehat{Y}_t,\widehat{Z}_t)$ is a concave function, $\quad \mathbf{P}\otimes dt -{\rm a.e.}$,

\item the transversality condition (TVC)

\begin{equation}\label{eq:tvc}
\operatorname*{\overline{lim}}_{t \rightarrow +\infty} \mathbf{E} \left(e^{-\beta t}\big<\widehat{X}_t - X_t,\widehat{Y}_t\big> \right) \leq 0,
\end{equation}
holds for every $X = X^{u(\cdot)}, u(\cdot)\in \mathcal{U}_{ad}$.
\end{enumerate}

Then $\widehat{u}(\cdot) = u^*(\cdot)$, i.e. $\widehat{u}(\cdot)$ is the optimal control strategy to the DCP.
\end{theorem}

\begin{remark}\label{rm:suff condition tvc }[Sufficient condition for TVC]
We will show that the conditions (H1)-(H6) immediately imply the TVC, which provides an easy-to-verify approach for the controller.\\

\noindent Indeed, we have shown that under (H1)-(H6), the solution processes to the FBSDE (\ref{eq:FBSDE})
\begin{align*}
(X,Y,Z) &= (X^{x,u(\cdot)},Y^{x,u(\cdot),X},Z^{x,u(\cdot),X}) \ \text{and} \\
(\widehat{X},\widehat{Y},\widehat{Z}) &= (\widehat{X}^{x,\widehat{u}(\cdot)},\widehat{Y}^{x,\widehat{u}(\cdot),\widehat{X}},\widehat{Z}^{x,\widehat{u}(\cdot),\widehat{X}}),
\end{align*}
are both in $\mathbf{L}^{2,-\beta}_{\mathcal{F}}\left(\mathbb{R}_+;\mathbb{R}^n \times \mathbb{R}^n \times \mathbb{R}^{n \times d} \right)$ for
$\beta > \operatorname{max}\left\{2\mu_1 + 2L^2, 2\mu_2 + 2M^2 \right\}.$ \\

\noindent Therefore, there exists a sequence $\left(T_n\right)_{n \in \mathbb{N}},\ T_n \nearrow +\infty$ as $n \rightarrow +\infty$ so that

\begin{equation}\label{eq:FBSDE a TVC}\nonumber
\left|\mathbf{E} \left( e^{-\beta T_n}\big<\widehat{X}_{T_n} - X_{T_n},\widehat{Y}_{T_n}\big>\right) \right| \leq
\frac{1}{2}\mathbf{E}\left( e^{-\beta {T_n}}|\widehat{X}_{T_n} - X_{T_n}|^2 \right) + \frac{1}{2} \mathbf{E}\left( e^{-\beta {T_n}}|\widehat{Y}_{T_n}|^2\right) \operatorname*{\rightarrow}_{n \rightarrow +\infty}0,
\end{equation}

\noindent which implies the TVC (\ref{eq:tvc}).

\end{remark}

\begin{remark}\label{rem time depend}
All results of section 4 holds also for time dependent coefficients $b: \mathbb{R}_+ \times \mathbb{R}^n \times U \rightarrow \mathbb{R}^n,
\sigma: \mathbb{R}_+ \times \mathbb{R}^n \times U \rightarrow \mathbb{R}^{n \times d}, f: \mathbb{R}_+ \times \mathbb{R}^n \times U \rightarrow \mathbb{R}$ under natural assumptions.
\end{remark}

\begin{remark}\label{rm:hamiltonian system}[Hamiltonian system] Throughout the paper we assume that the Hamiltonian takes the form (\ref{eq:hamiltonian}).
Nevertheless, one could easily verify that Theorem \ref{thm:postac PM1} holds also with the Hamiltonian function $ \mathcal{H}$ replaced by a different Hamiltonian function $H$ defined as
\begin{equation}\label{eq:hamiltonian 2}
H(x,u,y,z) = \big< b(x,u),y\big> + Tr(\sigma(x,u)'z) + f(x,u) = \mathcal{H}(x,u,y,z) + \beta \big<x,y\big>
\end{equation}

\noindent which enables us to rewrite the Forward-Backward Hamiltonian system as

\begin{align}\label{eq:Hamilt FBSDE}\nonumber
dX_t &= \nabla_y H(X_t,u_t, Y_t, Z_t)dt +\nabla_z H(X_t,u_t, Y_t, Z_t) dW_t, \quad \forall t \geq 0 \ {\rm a.s.}\\ \nonumber
  X_0 &= x \in \mathbb{R}^n, \\ \nonumber
-dY_t &= \Big( \nabla_x H(X_t,u_t, Y_t, Z_t) - \beta Y_t \Big)dt - Z_t dW_t, \quad \forall t \geq 0, \ {\rm a.s.} \\
\end{align}
\noindent Again, there is a correction term one has take into consideration, namely $- \beta Y_t$ in the driver of the BSDE.
\end{remark}

\section{Examples}
In this section we provide two illustrative examples of discounted control problems with well known solution. It will be shown that the approach using our maximum principle lead to the solution as well.

\subsection{Example 1 - Production planning problem}

The setting of the problem is taken from Borkar \cite{borkar} and it is in fact the classical LQ problem whose solution
is well known from e.g. dynamic programming, see \O ksendal \cite{oksendalSDE}, chapter 11. In the feedback Markov setting, the optimal control is obtained in the form of the observed (Markov) diffusion state process $X_t$ plugged into the solution of the associated Riccati ODE. \\

\noindent Let us consider a factory producing a single item. Let $X(\cdot)$ denote its inventory level as a function of time and $u(\cdot) \geq 0$ denote the production rate. $\eta$ denotes the constant demand rate and $x_1,u_1$ resp. the factory-optimal inventory level and production rate.\\
\noindent The inventory process is modeled as the controlled diffusion

\begin{equation}\label{eq: priklad 1 difuse}
dX_t =\big(u_t - \eta \big)dt + \sigma dW_t, \quad \forall t\geq 0,
\end{equation}

\noindent where $\sigma$ is a constant. The aim is to minimize over non-anticipative $u(\cdot)$ the discounted cost

\begin{equation}\label{eq: priklad 1funkcional}
J(u(\cdot)) = \mathbf{E} \int^{+\infty}_0 e^{-\beta t}\Big( c\big(u_t -u_1\big)^2+ h\big(X_t - x_1\big)^2 \Big)dt,
\end{equation}

\noindent where $c,h>0$ are known coefficients for the production cost and the inventory holding cost, resp. and $\beta >0$ is a discount factor.\\

\noindent The minimization task is converted to maximization by taking $-J(u(\cdot))$ as the functional.

\noindent Firstly, it is easy to see that (H1)-(H6) hold for $\mu_1 = \mu_2 = L = M = 0$ and therefore the DCP is meaningful for all $\beta \geq 0.$
The Hamiltonian of this control problem is

\begin{equation}\label{eq:priklad 1 hamiltonian}\nonumber
\mathcal{H}(x,u,y,z) = (u-\eta)y + \sigma z - c(u-u_1)^2 - h(x-x_1)^2 - \beta x y
\end{equation}

\noindent which is a concave function in $(x,u)$. The driver of the backward adjoint equation is obtained as the derivative of $\mathcal{H}$ w.r.t $x$, i.e.

\begin{equation}\label{eq:priklad 1 driver}\nonumber
h(x,u,y,z) = \frac{\partial}{\partial x}\mathcal{H}(x,u,y,z) = -2h(x-x_1) - \beta y
\end{equation}

\noindent and the associated BSDE is

\begin{equation}\label{eq:priklad 1 hamilt bsde inf}
-dY_t = \left( -2h(X_t-x_1) - \beta Y_t \right)dt - Z_tdW_t, \quad \forall t \geq 0, \ {\rm a.s.}
\end{equation}

\noindent To find the extremal (maximal) point of $\mathcal{H}$ we lay

\begin{equation}\label{eq:priklad 1 extrem}\nonumber
\frac{\partial}{\partial u}\mathcal{H}(x,u,y,z) = y - 2c(u-u_1) =0
\end{equation}

\noindent which leads to

\begin{equation}\label{eq:priklad 1 optimalni rizeni}
\widehat{u}_t = \frac{\widehat{Y}_t}{2c} + u_1.
\end{equation}

\noindent Further, assuming the solution to BSDE (\ref{eq:priklad 1 hamilt bsde inf}) in the feedback form

\begin{equation}\label{eq:priklad 1 zpetna vazba}
Y_t = \varphi_t X_t + \psi_t, \quad \forall t \geq 0 \ {\rm a.s.}
\end{equation}

\noindent for some deterministic functions $\varphi, \psi$ in $\mathcal{C}^1$ one can arrive to Riccati system of ODE's similarly as in
\cite{oksendal}.

\subsection{Example 2 - Optimal consumption rate}

The problem of optimal consumption rate is taken from \O ksendal \cite{oksendal priklad} where it is solved using a different definition of Hamiltonian. In our setting, we can, in fact, follow exactly the solution procedure step by step showing the same result. The main difference is that using our approach, one immediately knows whether the transversality condition holds just by examining the coefficients.

\noindent Consider an agent whose wealth evolves according to the following controlled bilinear SDE in $\mathbb{R}$

\begin{align}\label{eq: priklad difuse}
dX_t &=X_t\big(\mu - u_t \big)dt + \sigma X_t dW_t, \quad \forall t\geq 0,\nonumber \\
X_0 &= x_0 >0,
\end{align}

\noindent where $u(\cdot)$ is the consumption rate process, $\mu$ and $\sigma$ are some real constants (unlike in \cite{oksendal priklad} where they are assumed time dependent and bounded). The aim is to maximize over all  strictly positive controls $u(\cdot)$ bounded by some constant $K>0$ the discounted cost functional

\begin{equation}\label{eq: priklad funkcional}
J(u(\cdot)) = \mathbf{E} \int^{+\infty}_0 e^{-\beta t}\operatorname*{ln}\Big( u_t X_t\Big)dt,
\end{equation}

\noindent First we see that (H1)-(H6) hold for $\mu_1 = \mu_2 = \mu,\ M = L = |\sigma|$.
Further, one has to verify that the following condition holds

\begin{equation}\label{eq:predpoklad oksendal}
\mathbf{E}\int^{+\infty}_0 e^{-\beta t}||\nabla_x f(X_t,u_t)||^2dt <+\infty.
\end{equation}

\noindent To see it, realize that the solution to (\ref{eq: priklad difuse}) is a geometric Brownian motion

\begin{equation}\label{eq:predpoklad oksendal 2}
X_t = x_0 \operatorname*{exp}\left(\sigma W_t - \frac{1}{2} \sigma^2 t + \int^t_0 \left( \mu - u_s\right) ds \right),
\end{equation}
so the condition reads

\begin{align}\label{eq:predpoklad oksendal 3}
\mathbf{E}\int^{+\infty}_0 e^{-\beta t}\frac{1}{|X_t|^2}dt <+\infty.
\end{align}

\noindent Since we have

\begin{align}\label{eq:predpoklad oksendal 4}
\mathbf{E}\left[\frac{1}{|X_t|^2}\right]= 
x^{-2}_0 \mathbf{E} \Big[ \operatorname*{exp} \Big. & \Big. \left(\sigma^2 t  -2 \int^t_0 \left( \mu - u_s\right) ds \right)\operatorname*{exp}\left(-2\sigma W_t \right) \Big]\nonumber \\
&\leq x^{-2}_0  \operatorname*{exp}\left\{\left(3\sigma^2  -2\mu  +K \right)t\right\}  < +\infty,
\end{align}

\noindent Theorem \ref{thm:postac PM1} may be applied for all $\beta > \operatorname{max}\left\{2\mu + 2\sigma^2, K+3\sigma^2 - 2\mu \right\}$.

\noindent The Hamiltonian of this control problem is

\begin{equation}\label{eq:priklad hamiltonian}
\mathcal{H}(x,u,y,z) = x(\mu - u)y + x\sigma z + \operatorname*{ln}(xu) - \beta x y
\end{equation}

\noindent which is a concave function in $(x,u)$. The driver of the backward adjoint equation is obtained as the derivative of $\mathcal{H}$ w.r.t. $x$, i.e.

\begin{equation}\label{eq:priklad driver}\nonumber
\frac{\partial}{\partial x}\mathcal{H}(x,u,y,z) = (\mu - u - \beta)y + \sigma z +\frac{1}{x}.
\end{equation}

\noindent To find the maximal point of $\mathcal{H}$ we lay

\begin{equation}\label{eq:priklad extrem}\nonumber
\frac{\partial}{\partial u}\mathcal{H}(x,u,y,z) = -xy + \frac{1}{u} =0,
\end{equation}

\noindent which leads to

\begin{equation}\label{eq:priklad optimalni rizeni}
\widehat{u}_t = \frac{1}{\widehat{X}_t\widehat{Y}_t},
\end{equation}

\noindent and the associated BSDE is

\begin{equation}\label{eq:priklad hamilt bsde inf}
-d\widehat{Y}_t =  \left( (\mu - \widehat{u}_t - \beta)\widehat{Y}_t + \sigma \widehat{Z}_t +\frac{1}{\widehat{X}_t} \right) dt - \widehat{Z}_tdW_t, \quad \forall t \geq 0.
\end{equation}

\noindent The solution to BSDE (\ref{eq:priklad hamilt bsde inf}) for a general admissible control $u(\cdot)$ can be find along similar lines as in \cite{oksendal priklad}. We obtain

\begin{equation}\label{eq:reseni BSDE}
Y_t = \frac{1}{X_t \beta}.
\end{equation}

\noindent Finally, using (\ref{eq:priklad optimalni rizeni}) and (\ref{eq:reseni BSDE}) the optimal control is

\begin{equation}\label{eq: reseni}
\widehat{u}_t = \beta.
\end{equation}

\section{Controlled Stochastic logistic equation}
\noindent In this section we assume that the controlled dynamics is given by a stochastic differential equation with controlled logistic equation as a special case. It provides an example of a system with concave nonlinear coefficients. Here, the DCP is defined as follows. A one dimensional controlled state dynamics is given by

\begin{align}\label{eq: priklad 3 difuse 2}
dX_t &=X_t F\left(X_t, u_t \right)dt + \sigma \left(X_t \right) dW_t, \quad \forall t\geq 0, \nonumber \\
X_0 &=x_0,
\end{align}

\noindent where $u(\cdot)$ is a control process with values in
$U = \left[ \underline{U}, \overline{U}\right] \times \left[ \underline{V}, \overline{V}\right] \subset \mathbb{R}^2_+ = [0, +\infty)^2$ and $x_0 >0$ is a deterministic initial value. The functions $F: \mathbb{R} \times U \rightarrow \mathbb{R}$ and $\sigma: \mathbb{R} \rightarrow \mathbb{R}$ are continuously differentiable in $x$,$F$ is continuous in $(x,u)$, functions $b(x,u)= xF(x,u)$ and $\sigma(x)$ are concave in all their variables and we further assume that

\begin{itemize}
\item $\sigma(0) = 0,$ \\

\item $\sigma(\cdot)$ is globally Lipschitz, i.e. there exists $L > 0$ constant such that

\begin{equation}\nonumber
| \sigma(x_1) - \sigma(x_2)| \leq L |x_1 - x_2|,\ \forall (x_1, x_2) \in \mathbb{R}^2 ,
\end{equation}

\item there exists $M >0$ constant such that $|\sigma'(x)|\leq M,\ \forall x \in \mathbb{R}$, \\

\item there exists $\mu_1 \in \mathbb{R}$ such that

\begin{align}\nonumber
\left( x_1-x_2 \right) \big( x_1F(x_1,u) - x_2F(x_2,u)\big) \leq \mu_1 |x_1 - x_2|^2,\ \forall x_1, x_2 \geq 0, u \in  U,
\end{align}

\item there exists $\mu_2 \in \mathbb{R}$ such that

\begin{align}\nonumber
\big( y_1-y_2\big)^2 \left(F(x,u) + x\frac{\partial}{\partial x}F(x,u) \right) \leq \mu_2 |y_1 - y_2|^2,
\end{align}
$ \forall x>0, u \in U, (y_1, y_2) \in \mathbb{R}^2$.

\item there exist constants $r, R, C >0$ such that

\begin{align}\nonumber
-F(x,u) &\leq C,\ \forall x \in (0,r) , u \in  U, \\
F(x,u) &\leq C,\ \forall x \in (R,+\infty) , u \in  U.
\end{align}
\end{itemize}

\vspace{3mm}
\noindent Notice that the two functions $b(x,u) = xF(x,u)$ and $\sigma(x)$ fulfil the assumptions (H1) - (H6).
Assume further that

\begin{equation}\label{eq:omezeni F}
xF\left(x,\left(\underline{U}, \overline{V} \right)\right)  \leq xF(x,u) \leq xF\left(x,\left(\overline{U},\underline{V} \right)\right)  ,\ \forall x > 0,\ \forall u \in U,
\end{equation}
\noindent and denote the corresponding (constant) controls as $u_{lower} = \left(\underline{U}, \overline{V} \right)$ and $u_{upper} = \left(\overline{U},\underline{V} \right)$.\\

\noindent Note that the notation $F(\cdot,\cdot)$ in the drift in equation (\ref{eq: priklad 3 difuse 2}) covers situations of controlled drift such as: $xF(x,u) = ax(u_1 - bu_2x)$ for $u = (u_1, u_2)'$,
in which one can recognize the logistic equation with controlled both the saturation level of the diffusion and the speed of reaching this level.\\

\noindent The set of admissible controls $\mathcal{U}_{ad}$ is defined as

\begin{equation}\label{eq: priklad 3 pripustna rizeni}
\mathcal{U}_{ad} = \left\{ u: \mathbb{R}_+ \times \Omega \rightarrow U: u \in \mathbf{L}^2_{\mathcal{F},loc}\left(0,+\infty;U \right)   \right\}.
\end{equation}

\noindent The functional to be minimized over all admissible controls $u(\cdot)$ is

\begin{equation}\label{eq: priklad 3 funkcional}
J(u(\cdot)) = \mathbf{E} \int^{+\infty}_0 e^{-\beta t}\left( cX_t^2 + h |u_t|^2 \right)dt,
\end{equation}

\noindent where $c,h>0$ and $\beta >0$ is a discount factor. \\

\noindent First we start by showing that the solution to (\ref{eq: priklad 3 difuse 2}) for $x_0 >0$ stays in $(0,+\infty)$ for all $t\geq 0$ a.s., i.e.
it does not explode in finite time and the point $0$ is not attainable for the solution a.s.

\noindent We prove it in two steps. First, we prove this statement for the solution $X^{\widetilde{u}}=\widetilde{X}$ to (\ref{eq: priklad 3 difuse 2}) computed for an arbitrary but fixed constant control $\widetilde{u} \in U$. To proceed define for every $n \in \mathbb{N}$ and $\delta >0$ two stopping times
\begin{align}\label{eq:priklad 3 stop time}
\tau_n &=   \operatorname*{inf} \{t>0: |\widetilde{X}_t|\geq n \}, \nonumber \\
\tau_{\delta} &= \operatorname*{inf} \{t>0: |\widetilde{X}_t| \leq \delta \}.
\end{align}

\noindent Note that by a.s. continuity of $\widetilde{X}$, $\tau_n$ is a nondecreasing sequence of times almost surely and $\tau_{\delta}$ is a nondecreasing with $\delta$ decreasing to $0$.
Therefore there exist stopping times $\varepsilon$ (explosion time) and $\kappa$ (hitting time of $0$) so that
\begin{align}\label{eq:priklad 3 stop time limity}
\tau_n &\nearrow \varepsilon, \quad n \rightarrow +\infty,\ \text{a.s.}  \nonumber \\
\tau_{\delta} &\nearrow \kappa, \quad \delta \rightarrow 0_+,\ \text{a.s.}
\end{align}

\noindent Further, the infinitesimal generator of the Markov semigroup corresponding to $\widetilde{X}$ is given by the formula

\begin{equation}\nonumber\label{eq:priklad 3 generator}
\mathcal{L}f(x) = xF(x,\widetilde{u})\frac{d}{dx}f(x) + \frac{1}{2}\sigma^2(x)\frac{d^2}{dx^2}f(x),
\end{equation}
\noindent for $f \in \mathcal{C}^2(\mathbb{R}_+)$. Using the Lyapunov function $V(x) = 1 + \frac{1}{x} + x^2$ and assumptions on $\sigma(\cdot)$ and $F(\cdot,\cdot)$ one gets
\begin{equation}\nonumber \label{eq:priklad 3 generator ljap}
\mathcal{L}V(x) = xF(x,\widetilde{u})\left(-\frac{1}{x^2}+ 2x \right) +  \frac{1}{2}\sigma^2(x)\left(\frac{2}{x^3} + 2 \right) \leq K(1+ \frac{1}{x} +x^2) = K V(x),
\end{equation}

\noindent where $K = \operatorname*{max}\left\{L^2 + C, L^2 + 2\mu_1 \right\}$ for $x \in(0,r)$ and $K = \operatorname*{max}\left\{L^2, 2C + L^2 \right\}$ for $x>R$. \\

\noindent Further, $V(x)$ is radially unbounded, i.e.
\begin{align}\label{eq:priklad 3 radial neomezenost ljap}
q_R &= \operatorname*{inf}_{x\geq R} V(x) = 1+ R^2 \rightarrow +\infty \quad  \text{as} \quad R \rightarrow +\infty , \nonumber \\
q_{\delta} &= \operatorname*{inf}_{x\in (0,\delta]} V(x) = 1+ \frac{1}{\delta} + \delta^2 \rightarrow +\infty \quad \text{as} \quad \delta \rightarrow 0_+.
\end{align}

\noindent Now applying It\^o formula to $V\left(\widetilde{X}_{t\wedge \tau_n \wedge \tau_{\delta} }\right)$ and taking $\mathbf{E}(\cdot)$ one arrives at
\begin{align}\label{eq:priklad 3 Ito}
e^{Kt}V(x_0)&\geq \mathbf{E}\left[V\left(\widetilde{X}_{t\wedge \tau_n \wedge \tau_{\delta} }\right) \right] =
\mathbf{E}\left[\mathbf{1}_{\{\tau_n \wedge \tau_{\delta}\leq t\}} V\left(\widetilde{X}_{\tau_n \wedge \tau_{\delta}  }\right) \right]
+ \mathbf{E}\left[\underbrace{\mathbf{1}_{\{\tau_n \wedge \tau_{\delta}\geq t\}} V\left(\widetilde{X}_{t}\right)}_{\geq 0} \right] \nonumber\\
& \geq \mathbf{E}\left[\underbrace{\mathbf{1}_{\{\tau_n \wedge \tau_{\delta}\leq t\}}
\mathbf{1}_{\{\tau_{\delta} \leq \tau_n\}}}_{=\mathbf{1}_{\{\tau_{\delta} \leq t\}}}
 \underbrace{V\left(\widetilde{X}_{\tau_{\delta}  }\right)}_{\geq q_{\delta}} \right]
 +\mathbf{E}\left[\underbrace{\mathbf{1}_{\{\tau_n \wedge \tau_{\delta}\leq t\}}
\mathbf{1}_{\{\tau_{\delta} \geq \tau_n \}}}_{=\mathbf{1}_{\{\tau_n \leq t\}}}
 \underbrace{V\left(\widetilde{X}_{\tau_n  }\right)}_{\geq q_n} \right]\nonumber\\
 &\geq \mathbb{P}\left(\tau_{\delta} \leq t \right)q_{\delta} + \mathbb{P}\left(\tau_n \leq t \right)q_n.
\end{align}

\noindent Therefore, employing (\ref{eq:priklad 3 radial neomezenost ljap}) we have for each $t\geq 0$ fixed

\begin{align}\label{eq:priklad 3 konvergence}
&\mathbb{P}\left(\tau_n \leq t \right)\leq \frac{e^{Kt}V(x_0)}{q_n}\operatorname*{\longrightarrow}_{n \rightarrow +\infty} 0,\nonumber \\
&\mathbb{P}\left(\tau_{\delta}\leq t \right)\leq \frac{e^{Kt}V(x_0)}{q_{\delta}} \operatorname*{\longrightarrow}_{\delta \rightarrow 0_+} 0.
\end{align}

\noindent Taking into account (\ref{eq:priklad 3 stop time limity}) and (\ref{eq:priklad 3 konvergence}) we arrive at

\begin{equation}\label{eq:priklad 3 odhad pravdepodobnosti}
\operatorname*{max}\left\{\mathbb{P}\left(\varepsilon \leq t \right), \mathbb{P}\left(\kappa \leq t \right)   \right\} = 0, \quad \forall t\geq 0.
\end{equation}

\noindent Sending $t \rightarrow +\infty $, we finally obtain

\begin{equation}\label{eq:priklad 3 odhad pravdepodobnosti 2}
\operatorname*{max}\left\{\mathbb{P}\left(\varepsilon < +\infty \right), \mathbb{P}\left(\kappa < +\infty \right)   \right\} = 0,
\end{equation}

\noindent which means that with probability $1$ the solution $\widetilde{X}$ to equation (\ref{eq: priklad 3 difuse 2}) computed for some fixed constant control $\widetilde{u} \in U$ and for $x_0 >0$ does not explode and does not attain $0$ in finite time. \\

\noindent In the second step we show that for any admissible control $u(\cdot) \in \mathcal{U}_{ad}$, the solution $X = X^{u(\cdot)}$ to equation (\ref{eq: priklad 3 difuse 2}) satisfies

\begin{equation}\label{eq:priklad 3 sevreni reseni}
X^{u_{\text{lower}}}_t \leq X^{u(\cdot)}_t \leq X^{u_{\text{upper}}}_t,\ \forall t\geq 0,
\end{equation}

\noindent for the common initial condition $x_0>0$.\\

\noindent The result of the previous step holds in particular for $X^{u_{\text{lower}}}$ and $X^{u_{\text{upper}}}$.
For proving the unattainability of $0$, a comparison theorem will be the key tool. We employ Theorem 1.1 and Remark 1.1 in Ikeda, Watanabe, \cite{Ikeda} for special choice of functions

\begin{align}\label{eq:priklad 3 Ikeda}\nonumber
\beta_1(t,\omega) &= b_1(t,X^{u_{\text{lower}}}_t) = b_2(t,X^{u_{\text{lower}}}_t) = X^{u_{\text{lower}}}_t F\left(X^{u_{\text{lower}}}_t, u_{\text{lower}} \right),  \\
\beta_2(t,\omega) &= X^{u(\cdot)}_t F\left(X^{u(\cdot)}_t, u(t) \right).\nonumber
\end{align}

\noindent Using (\ref{eq:omezeni F}) we know that

\begin{equation}\label{eq:priklad 3 odhad reseni}
X^{u(\cdot)}(t \wedge \tau'_{\delta}) F\left(X^{u(\cdot)}(t \wedge \tau'_{\delta}), u_{\text{lower}} \right) \leq X^{u(\cdot)}(t \wedge \tau'_{\delta}) F\left(X^{u(\cdot)}(t \wedge \tau'_{\delta}), u(t \wedge \tau'_{\delta}), \right)
\end{equation}

\noindent for all $t\geq 0$, $x_0 >0$ and $\tau'_{\delta} = \operatorname*{inf} \{t>0: |X^{u(\cdot)}_t| \leq \delta \}$, $\delta >0$. Applying Theorem 1.1 and Remark 1.1 in \cite{Ikeda} one gets

\begin{equation}\label{eq:priklad 3 odhad reseni 2}
0< X^{u_{\text{lower}}}(s) \leq X^{u(\cdot)}(s),\forall s \in [0,\tau'_{\delta} ), \forall \delta >0.
\end{equation}

\noindent From the construction of stopping times $\tau'_{\delta}$, $\tau_{\delta}$ (the latter now defined in terms of the process $X^{u_{\text{lower}}}$), by (\ref{eq:priklad 3 odhad reseni 2}) and from the first step of the proof it follows that

\begin{equation}\label{eq:priklad 3 srovnani casu dosazeni 0}
0< \tau_{\delta} \leq \tau'_{\delta},
\end{equation}
 for all $\delta>0$. Since $\tau_{\delta} \rightarrow +\infty$ almost surely as $\delta \rightarrow 0_+$ we conclude that $X^{u(\cdot)}$ does not attain $0$ for all $t>0$ almost surely. Nonexplosion of $X^{u(\cdot)}$ can be proved similarly.\\

\noindent By Remark \ref{rm:oblast G} we conclude that Theorem \ref{thm:postac PM1} (or Theorem \ref{postac PM2}) can be applied for the above DCP. \\

\subsection{The controlled logistic equation}

\noindent  In the rest of the section we derive the form of the associated FBSDE and the Hamiltonian for a special choice of $F(\cdot,\cdot)$, namely
for $xF(x,u)=ax(1 - bx) + \gamma u,$ for $a,b >0, u \in [u_1,u_2], u_1 >0,\ \gamma \in \mathbb{R}$, i.e. the additive control is one dimensional process. The diffusion term is given by $\sigma(x) = \sigma x$ for some $\sigma >0$. The one dimensional controlled diffusion therefore evolves according to the SDE

\begin{align}\label{eq:logist aditiv}
dX_t &=a X_t \left( 1-bX_t \right)dt + \gamma u_t dt + \sigma X_t dW_t, \quad \forall t\geq 0,  \nonumber \\
X_0 &=x_0 > 0.
\end{align}

\noindent The functional to be minimized over all admissible controls $u(\cdot)$ is

\begin{equation}\label{eq:priklad 3 funkcional 2}
J(u(\cdot)) = \mathbf{E} \int^{+\infty}_0 e^{-\beta t}\left( cX_t^2 + h u^2_t\right)dt,
\end{equation}

\noindent where $c,h>0$ and $\beta >0$ is a discount factor. Again, we will maximize $-J(u(\cdot))$ and we see that the Hamiltonian is of the form

\begin{equation}\label{eq:priklad 3 hamiltonian}
\mathcal{H}(x,u,y,z) = a(x - bx^2)y + \gamma u y + \sigma x z - cx^2 - hu^2 - \beta x y,
\end{equation}

\noindent which is a concave function in $(x,u)$. The driver of the backward adjoint equation is obtained as the derivative of $\mathcal{H}$ w.r.t. $x$, i.e.

\begin{equation}\label{eq:priklad 3}\nonumber
\frac{\partial}{\partial x}\mathcal{H}(x,u,y,z) = (a -\beta)y - 2abxy +  \sigma z -2cx.
\end{equation}

\noindent Therefore, the associated BSDE reads

\begin{equation}\label{eq:priklad3 bsde}
-dY_t =  \left[ (a -\beta)Y_t - 2abX_tY_t +  \sigma Z_t -2cX_t  \right] dt - Z_t dW_t, \quad \forall t \geq 0.
\end{equation}

\noindent To find the maximal point of $\mathcal{H}$ w.r.t. $u$ we lay

\begin{equation}\label{eq:priklad extrem 2}\nonumber
\frac{\partial}{\partial u}\mathcal{H}(x,u,y,z) = \gamma y - 2hu =0,
\end{equation}

\noindent to obtain the argmax to the quadratic function in $u$ given by $\mathcal{H}$, i.e.

\begin{equation}\label{eq:priklad optimalni rizeni 2}
u_{max} = \frac{\gamma}{2h}y.
\end{equation}

\noindent Finally, the optimal control is obtained by studying the mutual position of $u_{max}$ and the interval $[u_1, u_2]$ which leads to the optimal control

\begin{equation}\label{eq:priklad optimalni rizeni 3}
\widehat{u}_t = \widetilde{u}\left(Y_t \right),
\end{equation}
where

\begin{equation}\label{eq:priklad3 tvar optimalniho rizeni}
\widetilde{u}\left(y\right) = \left\{
  \begin{array}{l l}
    &u_1,\ y \leq \frac{2h}{\gamma}u_1, \\ \\
    & \frac{\gamma}{2h}y,\ y \in \left[\frac{2h}{\gamma}u_1, \frac{2h}{\gamma}u_2 \right], \\ \\
    & u_2,\  y \geq \frac{2h}{\gamma}u_2.
  \end{array} \right.
\end{equation}

\noindent It can be easily seen that $\widehat{u}(\cdot)$ is bounded by $u_2$, continuous and nondecreasing. The last thing which is not
obvious is that $\widehat{u}_t$ is indeed an admissible control, i.e. that FBSDE (\ref{eq:logist aditiv}), (\ref{eq:priklad3 bsde}) admits a unique solution
with $\widehat{u}_t$ plugged. When this is shown then according to Theorem \ref{thm:postac PM1} (or Theorem \ref{postac PM2}) $\widehat{u}_t$ is optimal.

The corresponding FBSDE in the present case reads

\begin{align}\label{eq: FBSDE 3}
dX_t &=a X_t \left( 1-bX_t \right)dt + \gamma \widetilde{u}\left(Y_t \right) dt + \sigma X_t dW_t,  \nonumber \\
-dY_t &=  \left[ (a -\beta)Y_t - 2abX_tY_t +  \sigma Z_t -2cX_t  \right] dt - Z_t dW_t,  \nonumber \\
X_0 &=x_0 > 0.
\end{align}

\noindent FBSDE with the non-Lipschitz term $-2abxy$ in the driver of the backward part are not covered (up to authors' knowledge) with the existing theory. \\
\noindent We start with the local uniqueness theorem proved for the solution in the space $\mathbf{L}^{2}_{\mathcal{F},loc}\left(\mathbb{R}_+;\mathbb{R}^3 \right)$ which is sufficient for finite time horizon problem.

\begin{lemma}[Local uniqueness]\label{l:lok jednoznacnost}
Let $\left(X^i_t, Y^i_t, Z^i_t \right)_{t \in [t_0, t_0 + \delta]},\ i =1,2$ be two solution to FBSDE (\ref{eq: FBSDE 3}) for some $\delta >0,\ t_0 \geq 0$ such that
$X^1_{t_0} = X^2_{t_0}$ and $\ Y^1_{t_0+ \delta} = Y^2_{t_0+ \delta}$. Then

\begin{equation}\label{eq: FBSDE 3 unique}\nonumber
\left(X^1_t, Y^1_t, Z^1_t \right) = \left(X^2_t, Y^2_t, Z^2_t \right), \ \forall t \in [t_0, t_0 + \delta],\ \mathbb{P}- a.s.
\end{equation}
\end{lemma}

\begin{proof}[Proof:]
Denoting $\widehat{X}_t = X^2_t - X^1_t,\ \widehat{Y}_t = Y^2_t - Y^1_t,\ \widehat{Z}_t = Z^2_t - Z^1_t,\ \forall t \in [t_0, t_0 + \delta]$, applying It\^ o formula and using standard estimates one obtains

\begin{align}\label{eq: FBSDE 3 unique 1}\nonumber
\operatorname*{sup}_{t \in [t_0, t_0 + \delta]} \mathbf{E}|\widehat{X}_t |^2 &\leq C \left[\mathbf{E}|\widehat{X}_{t_0} |^2 + \mathbf{E}|\widehat{Y}_{t_0+ \delta} |^2 \right]
\exp\left(K \delta + K \mathbf{E}\int^{t_0+ \delta}_{t_0}|Y^1_s|^2 ds\right), \nonumber \\
\operatorname*{sup}_{t \in [t_0, t_0 + \delta]} \mathbf{E}|\widehat{Y}_t |^2 &+ \left(1 - \frac{\sigma}{\varepsilon_1}\right)\mathbf{E}\int^{t_0+ \delta}_{t_0}|\widehat{Z}_s|^2 ds
\leq C \left[\mathbf{E}|\widehat{Y}_{t_0+ \delta} |^2 + \mathbf{E}\int^{t_0+ \delta}_{t_0}|\widehat{X}_s |^2 ds \right]\nonumber \\
\times &\exp\left(K \delta + K \mathbf{E}\int^{t_0+ \delta}_{t_0}|\widehat{X}_t |^2 dt\right),
\end{align}

\noindent where $C,K > 0$ are some constants and $\varepsilon_1>0 $ is sufficiently small.

\noindent In a similar way one obtains the estimates with $\operatorname{sup}(\cdot)$ inside the expectation

\begin{align}
\left(1 - \varepsilon_2 \frac{C_1}{2}\right) \mathbf{E}\left[\operatorname*{sup}_{t \in [t_0, t_0 + \delta]}|\widehat{X}_t |^2 \right]
&\leq C \left[\mathbf{E}|\widehat{X}_{t_0} |^2 + \mathbf{E}|\widehat{Y}_{t_0+ \delta} |^2 \right]
\exp\left(K \delta + K \mathbf{E}\int^{t_0+ \delta}_{t_0} |Y^1_s|^2 ds\right), \label{eq:odhad x} \\ \nonumber \\
\left(1 - 2 C_1\varepsilon_3 \right)\mathbf{E}\left[\operatorname*{sup}_{t \in [t_0, t_0 + \delta]}|\widehat{Y}_t |^2 \right]
&+ \left(1 -  \frac{\sigma}{\varepsilon_1} - \frac{2c}{\varepsilon_3}\right) \mathbf{E}\int^{t_0+ \delta}_{t_0}|\widehat{Z}_s|^2 ds
\leq C  \left[\mathbf{E}|\widehat{X}_{t_0} |^2 + \mathbf{E}|\widehat{Y}_{t_0+ \delta} |^2 \right]  \nonumber \\
+ &C  \mathbf{E}\int^{t_0+ \delta}_{t_0} |Y^1_s|^2 \cdot |\widehat{X}_s|^2 ds,\label{eq:odhad yz}
\end{align}
where again $C,K > 0$ are some constants, $C_1> 0$ is the constant from Burkholder-Davis-Gundy inequality and $\varepsilon_1, \varepsilon_2, \varepsilon_3>0 $ are some sufficiently small numbers.

\noindent Since $\widehat{X}_{t_0} = \widehat{Y}_{t_0+ \delta} = 0$ uniqueness of the process $X$ follows from (\ref{eq:odhad x}). Uniqueness of the processes $Y,Z$ then follows from (\ref{eq:odhad yz}).
\end{proof}

\vspace{3mm}

\noindent To sketch the construction of the solution we first introduce for each $n \in \mathbb{N}$ the approximated equation

\begin{align}\label{eq: FBSDE 3 approx}
dX^n_t &=a X^n_t \left( 1-bX^n_t \right)dt + \gamma \widetilde{u}\left(Y^n_t \right) dt + \sigma X^n_t dW_t,  \nonumber \\
-dY^n_t &=  \left[ (a -\beta)Y^n_t - 2ab(X^n_t \wedge n )Y^n_t +  \sigma Z^n_t -2c(X^n_t \wedge n )  \right] dt - Z^n_t dW_t,  \nonumber \\
X^n_0 &=x_0 > 0.
\end{align}
\noindent The second equation in the system (\ref{eq: FBSDE 3 approx}) has Lipschitzian r.h.s. and it is well known (Yin, \cite{yin}) that it admits a unique solution $\left(X^n_t, Y^n_t, Z^n_t \right)_{t \geq 0}$ in $\mathbf{L}^{2, -\beta}_{\mathcal{F}}\left(\mathbb{R}_+;\mathbb{R}^3 \right)$
for each $n \in \mathbb{N}$ and for some exponential weight $\beta >0$ specified later.
We also note that the approximate drivers $h_n(x,y,z) = (a -\beta)y - 2ab(x \wedge n )y +  \sigma z -2c(x \wedge n )$ converge pointwise to the original driver $h(x,y,z) = (a -\beta)y - 2abxy +  \sigma z -2cx$ as $n \rightarrow +\infty$ and that the solution process $X$ inherits all the properties (nonexplosion etc.) proved above using the Lyapunov function. \\

\noindent Using standard techniques it is not difficult to derive the following uniform estimates (in $n$) of the process $X$ for
$\beta > \sigma^2 + 2a + \frac{C_1\sigma}{\varepsilon_3} + \max\left\{\gamma u_2, \frac{\gamma \varepsilon}{2} \right\}$

\begin{align}\label{eq:apriori odhady x approx}
\operatorname*{sup}_{t \geq 0}\mathbf{E}\left[e^{-\beta t}|X^n_t |^2 \right]  +  \left(1 - C_1 \sigma \varepsilon_3 \right) \mathbf{E}\left[\operatorname*{sup}_{t \geq 0} e^{-\beta t}|X^n_t |^2 \right]
&\leq x^2_0 + \frac{\gamma u_2}{\beta} < +\infty, \\ \nonumber \\
\operatorname*{sup}_{t \geq 0}\mathbf{E}\left[e^{-\beta t}|\widehat{X}_t |^2 \right]  +  \left(1 - C_1 \sigma \varepsilon_3 \right) \mathbf{E}\left[\operatorname*{sup}_{t \geq 0} e^{-\beta t}|\widehat{X}_t |^2 \right]
& \leq \frac{\gamma^3}{h^2 \varepsilon} \mathbf{E}\int^{+ \infty}_{0}e^{-\beta s} |\widehat{Y}_s|^2 ds,
\end{align}
where $\varepsilon, \varepsilon_3>0 $ are some sufficiently small numbers and $\widehat{X} = X^m - X^n,\ \widehat{Y} = Y^m - Y^n$ for each $m,n \in \mathbb{N}$. \\

\noindent Now, let $n,N \in \mathbb{N}$ be arbitrary and $t \geq 0$. Using previous estimates and Chebyshev inequality one can examine the boundedness of the process $X$ in probability
(uniformly in $t$ on any finite time interval $[0,T]$ for each $T>0$ fixed) in the following meaning

\begin{equation}\label{eq:apriori odhady x approx v psti 2}
\mathbb{P}\left( |X^n_t | > N\right) = \mathbb{P}\left(  e^{-\frac{\beta}{2} t}|X^n_t | > e^{-\frac{\beta}{2} t} N\right)
\leq \frac{\mathbf{E}\left( e^{-\beta t}|X^n_t |^2 \right)}{e^{-\beta t}N}\leq \frac{x^2_0 + \frac{\gamma u_2}{\beta}}{e^{-\beta T}N} \mathop{\longrightarrow }\limits_{N \to +\infty }  0,\\
\end{equation}
\noindent which will be a corner stone for the construction of the solution on finite time horizon. Therefore we consider the finite time horizon approximated equation

\begin{align}\label{eq: FBSDE 3 approx fin}
dX^n_t &=a X^n_t \left( 1-bX^n_t \right)dt + \gamma \widetilde{u}\left(Y^n_t \right) dt + \sigma X^n_t dW_t,\ t \in (0,T]  \nonumber \\
-dY^n_t &=  \left[ (a -\beta)Y^n_t - 2ab(X^n_t \wedge n )Y^n_t +  \sigma Z^n_t -2c(X^n_t \wedge n )  \right] dt - Z^n_t dW_t,\ t \in [0,T)  \nonumber \\
X^n_0 &=x_0 > 0,\nonumber \\
Y^n_T &= 0.
\end{align}

\noindent Employing classical results on BSDE we know that there is a unique the solution to (\ref{eq: FBSDE 3 approx fin}) in $\mathbf{L}^{2, -\beta}_{\mathcal{F}}\left([0,T];\mathbb{R}^3 \right)$ for each $T>0$ fixed and for each $n \in \mathbb{N}$. \\

\noindent Further, let $\Gamma_N$ be a cylindrical subset of $\mathbb{R}^3$ of the form

\begin{equation}\label{eq:cylindr}\nonumber
\Gamma_N = B(0,N)\times \mathbb{R} \times \mathbb{R},
\end{equation}
where $B(0,N)$ is a ball of radius $N \in \mathbb{N}$ with center in origin. Now we are ready to prove the consistency theorem.

\begin{theorem}\label{thm: konsistence reseni FBSDE}[Consistency] Let $T>0$ be arbitrary but fixed, $N \in \mathbb{N}$ be some (large enough) index and $m,p \in \mathbb{N},\ m,p>N$. Let
$\beta > \sigma^2 + 2a + \frac{C_1\sigma}{\varepsilon_3} + \max\left\{\gamma u_2, \frac{\gamma \varepsilon}{2} \right\}$
and let $\left(X^m_t, Y^m_t, Z^m_t \right)_{t \in [0,T]}$ and
$\left(X^p_t, Y^p_t, Z^p_t \right)_{t \in [0,T]}$ respectively be two solutions to FBSDE (\ref{eq: FBSDE 3 approx fin}) in $\mathbf{L}^{2, -\beta}_{\mathcal{F}}\left([0,T];\mathbb{R}^3 \right)$ computed for $n=m$, $n=p$ respectively. Then

\begin{equation}\label{eq:consistence 1}
\left(X^m_t, Y^m_t, Z^m_t \right)(\omega) = \left(X^p_t, Y^p_t, Z^p_t \right)(\omega),\ \forall t \in [0,T],\ \forall \omega \in \Omega_{\Gamma_N},
\end{equation}
where $\Omega_{\Gamma_N} = \left\{\omega \in \Omega: \left(X^m_t, Y^m_t, Z^m_t \right)(\omega)\in \Gamma_N, \left(X^p_t, Y^p_t, Z^p_t \right)(\omega) \in \Gamma_N,\ \forall t \in [0,T] \right\}$.
\end{theorem}
\noindent In other words, all the trajectories which stay in the cylinder $\Gamma_N$ coincide.

\begin{proof}[Proof:]
Let $\omega \in \Omega_{\Gamma_N}$ and denote $V^n_t = \left(X^n_t, Y^n_t, Z^n_t \right)$ for all $n \in \mathbb{N}$.
Then for $\left(V^p_t\right)(\omega)$ and $\left(V^m_t\right)(\omega)$, $m>p$, we obtain

\begin{align}
Y^p_t(\omega) &=  \int^T_t h_p\left(V^p_s(\omega) \right)\underbrace{\mathbf{1}_{\Gamma_p}\left(V^p_s(\omega) \right)}_{=1} ds  +
\int^T_t h_p\left(V^p_s(\omega) \right)\underbrace{\mathbf{1}_{\Gamma^c_p}\left(V^p_s(\omega) \right)}_{=0} ds - \left(\int^T_t Z^p_s dW_s\right)(\omega), \label{eq:consistence bsde p} \\
Y^m_t(\omega) &=  \int^T_t h_m\left(V^m_s(\omega) \right)\underbrace{\mathbf{1}_{\Gamma_m}\left(V^m_s(\omega) \right)}_{=1} ds  +
\int^T_t h_m\left(V^m_s(\omega) \right)\underbrace{\mathbf{1}_{\Gamma^c_m}\left(V^m_s(\omega) \right)}_{=0} ds \nonumber \\
&\hspace{9cm} - \left(\int^T_t Z^m_s dW_s\right)(\omega).\label{eq:consistence bsde m}
\end{align}
We also notice that

\begin{align}\label{eq:consistence driver}\nonumber
h_p(x,y,z)\mathbf{1}_{\Gamma_p}(x,y,z) &= h_m(x,y,z)\mathbf{1}_{\Gamma_p}(x,y,z) = h(x,y,z)\mathbf{1}_{\Gamma_p}(x,y,z)\ \text{and} \nonumber \\
h_p(x,y,z)\mathbf{1}_{\Gamma_N}(x,y,z) &= h_m(x,y,z)\mathbf{1}_{\Gamma_N}(x,y,z) = h_N(x,y,z)\mathbf{1}_{\Gamma_N}(x,y,z).\nonumber \\
\end{align}
Since $\Gamma_N \subset \Gamma_p \subset \Gamma_m$ for $m>p>N$, $\left(V^p_t\right)(\omega)$ also satisfies the same integral relation (\ref{eq:consistence bsde m}) and the forward equation
of the system, i.e. $\left(V^p_t\right)(\omega)$ solves (\ref{eq:consistence bsde m}) as well, and by the local uniqueness (Lemma \ref{l:lok jednoznacnost} with $n = m$) it follows that
$\left(V^p_t\right)(\omega) = \left(V^m_t\right)(\omega),\ \forall t \in [0,T]$. \\

\noindent Further, we can prove along similar lines that $\left(V^p_t\right)(\omega), \left(V^m_t\right)(\omega)$ solve (\ref{eq: FBSDE 3 approx fin}) for $n=N$.
Putting both arguments together we see that all the solution trajectories that stay in some cylinder (for some index) solve also all the equation with
larger index and all such trajectories coincide.
\end{proof}

\begin{proposition}[Construction of the solution]\label{p:existence reseni konec horiz}
Under the assumptions of Theorem \ref{thm: konsistence reseni FBSDE} there exists a unique solution $\left(X_t, Y_t, Z_t \right)_{[0,T]} \in \mathbf{L}^{2, -\beta}_{\mathcal{F}}\left(\mathbb{R}_+;\mathbb{R}^3 \right)$
of the finite horizon FBSDE

\begin{align}\label{eq: FBSDE fin horiz}
dX_t &=a X_t \left( 1-bX_t \right)dt + \gamma \widetilde{u}\left(Y_t \right) dt + \sigma X_t dW_t,  \nonumber \\
-dY_t &=  \left[ (a -\beta)Y_t - 2abX_tY_t +  \sigma Z_t -2cX_t  \right] dt - Z_t dW_t,  \nonumber \\
X_0 &=x_0 > 0,\nonumber \\
Y_T &= 0.
\end{align}

\end{proposition}

\begin{proof}[Proof:]
\noindent Keeping the notation from the previous proof notice that

\begin{equation}\label{eq:konstrukce reseni}\nonumber
\Omega_{\Gamma_N} = \Omega_{p,N} \cap \Omega_{m,N},
\end{equation}

\noindent where $\Omega_{n,N} = \left\{\omega \in \Omega: \left(X^n_t, Y^n_t, Z^n_t \right)(\omega) \in \Gamma_N,\ \forall t \in [0,T] \right\}$, for $n \in \mathbb{N}$ and due to estimate (\ref{eq:apriori odhady x approx v psti 2}) we have that $\mathbb{P}\left( \Omega_{p,N} \right) \mathop{\longrightarrow }\limits_{N \to +\infty } 1$ uniformly in $p$. Thus we put for any fixed $p > N$

\begin{equation}\label{eq:konstrukce reseni omega}
\widetilde{\Omega}= \bigcup_{N \in \mathbb{N}} \Omega_{p,N},\ \text{with} \ \mathbb{P}\left( \widetilde{\Omega} \right) = 1.
\end{equation}
The unique limit solution process to FBSDE

\begin{align}\label{eq: FBSDE 3 fin}
dX_t &=a X_t \left( 1-bX_t \right)dt + \gamma \widetilde{u}\left(Y_t \right) dt + \sigma X_t dW_t,\ t \in (0,T]  \nonumber \\
-dY_t &=  \left[ (a -\beta)Y_t - 2abX_tY_t +  \sigma Z_t -2cX_t  \right] dt - Z_t dW_t,\ t \in [0,T)  \nonumber \\
X_0 &=x_0 > 0,\nonumber \\
Y_T &= 0.
\end{align}

\noindent is obtained for all $t \in [0,T]$ as

\begin{equation}\label{eq:konstrukce reseni fin}
\left(X_t, Y_t, Z_t \right)(\omega) =
\left\{
	\begin{array}{ll}
		\left(X^N_t, Y^N_t, Z^N_t \right)(\omega),  & \mbox{for } \omega \in  \Omega_{p,N}, \\
		\left(X^{N+1}_t, Y^{N+1}_t, Z^{N+1}_t \right)(\omega), & \mbox{for } \omega \in  \Omega_{p,N+1}\backslash \Omega_{p,N},
	\end{array}
\right.
\end{equation}
for each $N \in \mathbb{N}$. We stress that this construction is independent of the choice of $p$ due to the consistency theorem.
\end{proof}

\noindent Construction of the solution to the original FBSDE (\ref{eq: FBSDE 3}) now follows the lines of analogous proof in Pardoux \cite{darling pardoux} with the process $\left(X_t, Y_t, Z_t \right)$ from Proposition
\ref{p:existence reseni konec horiz} as its finite time horizon approximation. That is, the solution to (\ref{eq: FBSDE 3}) is obtained as a limit of solutions $\left(X^n_t, Y^n_t, Z^n_t \right)_{t\geq 0}$ where

\begin{align}\label{eq: FBSDE 3 infin}
dX^n_t &=a X^n_t \left( 1-bX^n_t \right)dt + \gamma \widetilde{u}\left(Y^n_t \right) dt + \sigma X^n_t dW_t,\ t \in (0,n]  \nonumber \\
-dY^n_t &=  \left[ (a -\beta)Y^n_t - 2abX^n_tY^n_t +  \sigma Z^n_t -2cX^n_t  \right] dt - Z^n_t dW_t,\ t \in [0,n)  \nonumber \\
X^n_0 &=x_0 > 0,\nonumber \\
X^n_t &=X^n_n,t \geq n, \nonumber \\
Y^n_t &= 0, t \geq n \nonumber \\
Z^n_t &= 0, t \geq n.
\end{align}


\begin{thebibliography}{9}


\bibitem{oks delay}
Agram N., Haadem S., \O ksendal B. and Proske F.
\newblock {\em A maximum principle for infinite horizon delay equations.}
\newblock {\em In preparation.}

\bibitem{hussein}
Al-Hussein A.
\newblock {\em Maximum Principle for Controlled Stochastic Evolution Equations.}
\newblock {\em Int. Journal of Math. Analysis.} {\bf 4}, 1447-1464, 2010

\bibitem{bahlali mezerdi}
Bahlali S. and Mezerdi B.
\newblock {\em A General stochastic maximum principle for singular control problems.}
\newblock {\em Electronic Journal of Probability} {\bf 10}, 988-1004, 2005

\bibitem{bensoussan}
Bensoussan A.
\newblock {\em Maximum principle and dynamic programming approaches of the optimal control of partially observed diffusions.}
\newblock {\em Stochastics.} {\bf 9}, 169-222, 1983.

\bibitem{bismut}
Bismut J.-M.
\newblock {\em Conjugate convex functions in optimal stochastic control.}
\newblock {\em J. Math. Anal. Appl.} {\bf 44}, 384-404, 1973.

\bibitem{borkar}
Borkar V. S.
\newblock {\em Controlled diffusion processes.}
\newblock {\em Probability Surveys.} {\bf 2}, 213-244, 2005.

\bibitem{haussmann}
Cadenillas A. and Haussmann U. G.
\newblock {\em The stochastic maximum principle for a singular control problem.}
\newblock {\em Stochastics Rep.} {\bf 49}, 211-237, 1994.

\bibitem{chikodza}
Chikodza E.
\newblock {\em COMBINED SINGULAR AND IMPULSE CONTROL FOR JUMP DIFFUSIONS.}
\newblock {\em SAMSA Journal of Pure and Applied Mathematics.} {\bf 3}, 29-57, 2008

\bibitem{dufour}
Dufour, F. and Miller, B.
\newblock {\em Maximum principle for singular stochastic control problems.}
\newblock {\em SIAM J. Control Optim.} {\bf 45}, 668-698, 2006

\bibitem{friedman}
Friedman A.
\newblock {\em Stochastic Differential Equations and Applications: Volume 1.}
\newblock {\em Academic Press},  1975


\bibitem{fuhrman tessitore hu}
Fuhrman M, Hu Y. and Tessitore G.
\newblock {\em Stochastic maximum principle for optimal control of SPDEs.}
\newblock {\em Comptes Rendus Mathematique.} {\bf 350}, 683-688, 2012

\bibitem{oksendal priklad}
Haadem S., Proske F. and \O ksendal B.
\newblock {\em Maximum principles for jump diffusion processes with infinite horizon.}
\newblock {\em http://arxiv.org/pdf/1206.1719.pdf}

\bibitem{Ikeda}
Ikeda, N. and Watanabe. S.
\newblock {\em A COMPARISON THEOREM FOR SOLUTIONS OF STOCHASTIC DIFFERENTIAL EQUATIONS AND ITS APPLICATIONS.}
\newblock {\em Osaka J. Math.} {\bf 14}, 619-633, 1977.

\bibitem{kushner}
Kushner H. J.
\newblock {\em Necessary conditions for continuous parameter stochastic optimization problems.}
\newblock {\em SIAM J. Control.} {\bf 10}, 550-565, 1972.

\bibitem{oksendalSDE}
\O ksendal B.
\newblock {\em Stochastic Differential Equations. An Introduction with Applications, 4th edition.}
\newblock {\em Berlin, Springer-Verlag 1995.}, ISBN 3-540-60243-7 (Universitext)

\bibitem{oksendal singular levy}
\O ksendal B. and Sulem A.
\newblock {\em Singular stochastic control and optimal stopping with partial information of Itô–Lévy processes.}
\newblock {\em SIAM J. Control Optim.} {\bf 50}, 2254-2287, 2012

\bibitem{oksendal}
\O ksendal B., Sulem A. and Framstad N. C.
\newblock {\em A sufficient stochastic maximum principle for optimal control of jump diffusions and applications to finance. }
\newblock {\em J. Optimization Theory and Applications}, {\bf 121}, 77-98, 2004.
\newblock Errata: {\em J. Optimization Theory and Applications} {\bf 124}, 511-512, 2005.

\bibitem{oks sulem zhang}
\O ksendal B., Sulem A. and Zhang T.
\newblock {\em OPTIMAL CONTROL OF STOCHASTIC DELAY EQUATIONS AND TIME-ADVANCED BACKWARD STOCHASTIC DIFFERENTIAL EQUATIONS.}
\newblock {\em Int. Advances in Applied Probability.} {\bf 43}, 572-596, 2011

\bibitem{pardoux}
Pardoux E.
\newblock {\em BSDEs weak convergence and homogenizations of semilinear PDEs.}
\newblock {\em Nonlinear Analysis Differential Equations and Control.}
\newblock {\em Clark, F.H., Stern, R.J. (Eds.), Kluwer Academic, Dordrecht}, 503-549, 1999.


\bibitem{darling pardoux}
Pardoux E., R. W. R. Darling
\newblock {\em Backwards SDE with Random Terminal Time and Applications to Semilinear Elliptic PDE.}
\newblock {\em The Annals of Probability.} Vol. 25, No. 3, 1135-1159, 1997.

\bibitem{pardouxpeng}
Pardoux E., Peng S. G.
\newblock {\em Adapted solution of a backward stochastic differential equation.}
\newblock {\em Systems \& Control Letters.} {\bf 14}, 55-61, 1990.

\bibitem{peng}
Peng S. G.
\newblock {\em A general stochastic maximum principle for optimal control problems.}
\newblock {\em SIAM J. Control Optim.} {\bf 28}, 966-979, 1990.

\bibitem{peng shi}
Peng S. G., Shi Y.
\newblock {\em Infinite horizon forward–backward stochastic differential equations.}
\newblock {\em Stochastic Process. Appl.} {\bf 85}, 75–92, 2000.


\bibitem{pham}
Pham, H.
\newblock {\em Continuous-time stochastic control and optimization with financial applications.}
\newblock {\em Stochastic Modelling and Applied Probability} {\bf 61}, Springer-Verlag, Berlin, 2009.


\bibitem{tang}
Tang S. J. and Li X. J.
\newblock {\em Necessary conditions for optimal control of stochastic systems with random jumps.}
\newblock {\em SIAM J. Control Optim.} {\bf 32}, 1447-1475, 1994.

\bibitem{yin}
Yin J.
\newblock {\em On solutions of a class of infinite horizon FBSDE's.}
\newblock {\em Statistics and Probability Letters.} {\bf 78}, 2412-2419, 2008.



\bibitem{wu}
Wu Z.
\newblock {\em Forward-Backward Stochastic Differential Equations with Stopping Time.}
\newblock {\em ACTA MATHEMATICA SCIENTIA.} {\bf 24}, 91-99, 2004.


\bibitem{wu zhang}
Wu Z. and Zhang F.
\newblock {\em Maximum Principle for Stochastic Recursive Optimal Control Problems Involving Impulse Controls.}
\newblock {\em Abstract \& Applied Analysis.} {\bf 32}, 1-16, 2012


\bibitem{zhou near opt}
Zhou X.Y.
\newblock {\em STOCHASTIC NEAR-OPTIMAL CONTROLS: NECESSARY AND SUFFICIENT CONDITIONS FOR NEAR-OPTIMALITY.}
\newblock {\em SIAM J. Control Optim.} {\bf 36}, 929-947, 1998











\end{thebibliography}
\end{document}